\documentclass[11pt, notitlepage]{article}
\usepackage{amssymb,amsmath,comment}
\catcode`\@=11 \@addtoreset{equation}{section}
\def\thesection{\arabic{section}}

\def\theequation{\thesection.\arabic{equation}}
\catcode`\@=12
\usepackage{colortbl}%
\usepackage[mathscr]{eucal}
\usepackage{epsf}
\usepackage{a4wide}

\newcommand{\ds} {\displaystyle}

\newcommand{\pa} {\partial}
\newcommand{\al} {\alpha}

\newcommand{\Om} {\Omega}
\newcommand{\ra} {\rightarrow}

\newcommand{\De} {\Delta}
\newcommand{\la} {\lambda}

\newcommand{\noi} {\noindent}

\newcommand{\mb} {\mathbb}
\newcommand{\mc} {\mathcal}

\usepackage[all]{xy}
\catcode`\@=11
\def\theequation{\@arabic{\c@section}.\@arabic{\c@equation}}
\catcode`\@=12

\def\QED{\hfill {$\square$}\goodbreak \medskip}

\newtheorem{Theorem}{Theorem}[section]
\newtheorem{Lemma}[Theorem]{Lemma}
\newtheorem{Proposition}[Theorem]{Proposition}

\newtheorem{Remark}[Theorem]{Remark}
\newtheorem{Definition}[Theorem]{Definition}

\begin{document}

{\vspace{0.01in}}

\title
{ \sc Fractional Choquard Equation with Critical Nonlinearities}

\author{
    ~~ T. Mukherjee\footnote{e-mail: tulimukh@gmail.com} ~ and K. Sreenadh \footnote{ sreenadh@gmail.com}\\
    Department of Mathematics\\ Indian Institute of Technology Delhi\\
Hauz Khaz, New Delhi-110016, India.
 }

\date{}

\maketitle

\begin{abstract}

\noi In this article, we study the Brezis-Nirenberg type problem of  nonlinear Choquard equation involving a fractional Laplacian
\[ (-\De)^s u = \left( \int_{\Om}\frac{|u|^{2^*_{\mu,s}}}{|x-y|^{\mu}}\mathrm{d}y \right)|u|^{2^*_{\mu,s}-2}u +\la u \; \text{in } \Om,\]
where $\Om $ is a bounded domain in $\mathbb R^n$ with Lipschitz  boundary, $\la $ is a real parameter, $s \in (0,1)$, $n >2s$ and $2^*_{\mu,s}= (2n-\mu)/(n-2s)$ is the critical exponent in the sense of Hardy-Littlewood-Sobolev inequality. We obtain some existence, multiplicity, regularity and nonexistence results for solution of the above equation using variational methods.
\medskip

\noi \textbf{Key words:}  Fractional Laplacian, Brezis-Nirenberg problem, Choquard equation, Critical exponent.

\medskip

\noi \textit{2010 Mathematics Subject Classification:} 35R11, 35R09, 35A15.

\end{abstract}

\section{Introduction}
In the present paper, we study the existence of solutions of the following doubly nonlocal fractional elliptic equation:
\begin{equation*}
(P_\la): (-\De)^su = \left( \int_{\Om}\frac{|u|^{2^*_{\mu,s}}}{|x-y|^{\mu}}\mathrm{d}y \right)|u|^{2^*_{\mu,s}-2}u +\la u \; \text{ in } \Om, \; \; u =0 \; \text{ in } \mathbb R^n\setminus \Om,
\end{equation*}
where $\Om$ is a bounded domain in $\mathbb R^n$ with Lipschitz boundary, $\la $ is a real parameter, $s\in (0,1)$, $n > 2s$, $2^*_{\mu,s}= (2n-\mu)/(n-2s)$ and $(-\De)^s$ is the fractional Laplace operator defined  as
$$ (-\De)^s u(x) = -\mathrm{P.V.}\int_{\mb R^n} \frac{u(x)-u(y)}{\vert x-y\vert^{n+2s}}\,\mathrm{d}y$$
({up to a normalizing constant}), where $\mathrm{P.V.}$ denotes the Cauchy principal value.
The fractional power of Laplacian is the infinitesimal generator of L$\acute{e}$vy stable diffusion process and arise in anomalous diffusion in plasma, population dynamics, geophysical fluid dynamics, flames propagation, chemical reactions in liquids and American options in finance. For more details, we refer to \cite{da,gl}. {Problems of the type $(P_\la)$  are inspired by the Hardy-Littlewood-Sobolev inequality:
\begin{equation}\label{new1}
\left(\int_{\mb R^n} \int_{\mb R^n}\frac{|u(x)|^{2^*_{\mu,s}}|u(y)|^{2^*_{\mu,s}}}{|x-y|^{\mu}}\mathrm{d}x\mathrm{d}y\right)^{\frac{1}{2^*_{\mu,s}}} \leq C ^{\frac{1}{2^*_{\mu,s}}}|u|^2_{2^*_s}, \; \text{for all} \; u\in H^s(\mathbb{R}^n).\end{equation}
where $C=C(n,\mu)$ is a positive constant and $2^{*}_{s}=\frac{2n}{n-2s}.$\\

\noi In the local case $s=1$, authors in \cite{BJS} studied the existence of of ground states for the nonlinear Choquard equation
\begin{equation}\label{cho1}
 -\De u + V(x)u = \left( \frac{1}{|x|^{\mu}} * |u|^p\right)|u|^{p-2}u \; \text{ in } \mb R^n,
 \end{equation}
where $p>1$ and $n\ge 3$. }
In the case when $p=2$  and $\mu=1$, S. Pekar \cite{Pekar} used this equation to describe the quantum theory of a polaron at rest and P. Choquard \cite{Lieb} adopted it as an approximation to Hartree-Fock theory of one component plasma. In \cite{moroz1}, authors considered the existence of ground states under the assumptions of Berestycki-Lions type. With conditions on the potential $V$, problems of type \eqref{cho1} are also  studied in \cite{ANM,Acker}.\\

\noi In \cite{lb}, Lieb considered the problem of the form
 \[ -\De u+ u = (|x|^{\mu} * F(u))f(u) \; \text{in}\; \mb R^n ,\]
 where  $f(t)$ is critical growth nonlinearity such that $|tf(t)| \leq C||t|^2 + |t|^{\frac{2n-\mu}{n-2s}}|$, for $t \in \mb R$, some constant $C>0$ and $F(t) = \int_0^{z}f(z) \mathrm{d}z$. Under some appropriate structure conditions on the nonlinearity $f$ author proved the existence and uniqueness (up to translations) of the ground state solutions. The existence of a sequence of  radially symmetric solutions was shown by Lions in \cite{Lions}. The nonlocal counterpart of this problem with fractional Laplacian has been studied in \cite{SGY}. A class of Schr\"{o}dinger equations with a generalized Choquard nonlinearity and fractional diffusion has been investigated in \cite{chosquas1}. Some existence, nonexistence and regularity results has been studied in \cite{chosquas2}. For more details, we  refer to \cite{wang,gv,clapp,AM1,AM2}.\\

\noi In the pioneering work of Brezis-Nirenberg \cite{breniren}, authors studied  the   critical exponent problem
\begin{equation*}
-\De u = |u|^{2^*-2}u + \la u \; \text{in} \; \Om, \; u = 0 \; \text{in} \; \partial \Om,
\end{equation*}
where $2^*=\frac{n+2}{n-2}$. They proved the existence of solutions for $\la>0, n>4$ by analysing the .local Palais-Smale sequences below the first critical level.
 In \cite{myang},   Gao and Yang established some existence results for the Brezis-Nirenberg type problem of the nonlinear Choquard equation
\begin{equation} \label{eqnew} -\De u = \left( \int_{\Om}\frac{|u|^{2^*_{\mu}}}{|x-y|^{\mu}}\mathrm{d}y \right)|u|^{2^*_{\mu}-2}u +\la u \; \text{in } \Om,\quad u=0\; \text{on}\; \pa\Om, \end{equation}
where $\Om$ is a bounded domain with LIpschitz boundary in $\mb R^n$ , $n \geq 3$, $\la$ is a parameter and $2^*_{\mu}= (2n-\mu)/(n-2)$. Here again, authors obtianed the existence results using mountain pass structure of the energy functional and and carefully analysing the local Palais-Smale sequences below the first critical level as in \cite{breniren}. \\

\noi Recently, many people studied the Brezis-Nirenberg type results for semilinear equations with fractional Laplacian, for details and recent works we  refer to \cite{BCSF,bn-serv,s3,bc,mb,mb1,ts1,ts2} and the references therein. In \cite{MK1,MK2}, the authors discuss recent developments in the description of anamolous diffusion via fractional dynamics and several fractional equations are obtained asymptotically from L\'{e}vy random walk models, extending Brownian walk models in a natural way. Particularly, in \cite{Laskin} a fractional Schr\"{o}dinger equation with local power type nonlinearity has been studied.\\

\noi In this paper, we  consider the nonlocal counterpart of the problem in \eqref{eqnew} namely $(P_\la)$. {Here, we study  the existence, multiplicity, regularity and nonexistence results  for  $(P_\la)$ in the spirit of \cite{myang}. We show several estimates while studying the compactness of  Palais-Smale sequences using the minimizers of the inequality in \eqref{new1}} and show the $L^\infty$ and $C^{0,\al}$ regularity for the solutions of $(P_\la).$ To the best of our knowledge, there is no paper considering the choquard equation with critical growth and fractional Laplacian. We aim at studying the existence and multiplicity of choquard equation with upper critical exponent $2^*_{\mu,s}= (2n-\mu)/(n-2s)$ on bounded domain in $\mb R^n, \; n>2s$ and answer completely to the question of existence, multiplicity and nonexistence of solutions. We are interested in the problem that how perturbation with a linear term along with double nonlocal terms affect the existence and multiplicity of the problem $(P_\la)$.  \\


\noi The paper is organized as follows: In section 2, we give the functional setting for the problem to use variational approach and state our main results. In section 3, we show that the weak limit of every bounded Palais-Smale sequence gives a weak solution for $(P_\la)$ by analyzing the Palais-Smale sequences below the critical level. In section 4, we give the proof of our first main theorem (when $n\geq 4s$) for the cases $\la \in (0,\la_1)$ and $\la \geq \la_1$ separately, where $\la_1$ is the first eigenvalue of $(-\De)^s$ with homogenous Dirichlet datum given in $\mb R^n \setminus \Om$. In section 5, we prove the existence result for $(P_\la)$ when $ 2s<n<4s $, that is we show that there exists $\bar{\la}>0$ such that for any $\la > \bar \la $, different from the eigenvalues of $(-\De)^s$,  $(P_\la)$ has a nontrivial solution. In section 6, we present the multiplicity results for $(P_\la)$. In section 7, we show some regularity result for weak solutions of $(P_\la)$. Finally, in section 8, we prove a non-existence result for $\la <0.$

\medskip

\section{Functional Setting and Main results}
\noi In \cite{bn-serv}, Servadei and  Valdinoci discussed the Dirichlet
boundary value problem for the fractional Laplacian using  variational methods.  Due to the nonlocalness of the fractional
Laplacian, they introduced the function space $(X_0,\|.\|_{X_0})$.
The space $X$ is defined as
\[X= \left\{u|\;u:\mb R^n \ra\mb R \;\text{is measurable},\;
u|_{\Om} \in L^2(\Om)\;
 \text{and}\;  \frac{(u(x)- u(y))}{ |x-y|^{\frac{n}{2}+s}}\in
L^2(Q)\right\},\]
\noi where $Q=\mb R^{2n}\setminus(\mc C\Om\times \mc C\Om)$ and
 $\mc C\Om := \mb R^n\setminus\Om$. The space X is endowed with the norm
 \[\|u\|_X = \|u\|_{L^2(\Om)} + \left[u\right]_X,\]
 where
 \[\left[u\right]_X= \left( \int_{Q}\frac{|u(x)-u(y)|^{2}}{|x-y|^{n+2s}}\,\mathrm{d}x\mathrm{d}y\right)^{\frac12}.\]
 Then we define $ X_0 = \{u\in X : u = 0 \;\text{a.e. in}\; \mb R^n\setminus \Om\}$. Also we have the Poincare type inequality:  there exists a constant $C>0$ such that $\|u\|_{L^{2}(\Om)} \le C [u]_X$, for all $u\in X_0$. Hence,  $\|u\|=[u]_X$ is a norm on $(X_0, \|.\|)$. Moreover, $X_0$ is a Hilbert space and $C_c^{\infty}(\Om)$ is dense in $X_0$.
   Note that the norm $\|.\|$ involves the interaction between $\Om$ and $\mb R^n\backslash\Om$.
 We denote $\|.\|=[.]_X$ for the norm in $X_0$.
 From the embedding results, we know that $X_0$ is continuously and compactly embedded in $L^r(\Om)$  when $1\leq r < 2^*_s$, where $2^*_s = 2n/(n-2s)$ and the embedding is continuous but not compact if $r= 2^*_s$. We define
 \begin{equation*}
 {S_s} = \inf_{u \in X_0\setminus \{0\}} \frac{\int_Q \frac{|u(x)-u(y)|^2}{|x-y|^{n+2s}}{\,\mathrm{d}x\mathrm{d}y}}{\left(\int_{\Om}|u|^{2^*_s}\,\mathrm{d}x\right)^{2/2^*_s}} .
 \end{equation*}
 The key point to apply variational approach for the problem $(P_\la)$ is the following well-known Hardy-Littlewood-Sobolev inequality.
 \begin{Proposition}\cite{Lieb}
 Let $t,r>1$ and $0<\mu<n $ with $1/t+\mu/n+1/r=2$, $f \in L^t(\mathbb R^n)$ and $h \in L^r(\mathbb R^n)$. There exists a sharp constant $C(t,n,\mu,r)$, independent of $f,h$ such that
 \begin{equation*}
 \int_{\mb R^n}\int_{\mb R^n} \frac{f(x)h(y)}{|x-y|^{\mu}}\mathrm{d}x\mathrm{d}y \leq C(t,n,\mu,r)|f|_t|h|_r.
 \end{equation*}
 \end{Proposition}
In general, let $f = h= |u|^q$ then by Hardy-Littlewood-Sobolev inequality we get,
\[ \int_{\mb R^n}\int_{\mb R^n} \frac{|u(x)|^q|u(y)|^q}{|x-y|^{\mu}}\mathrm{d}x\mathrm{d}y\]
 is well defined if $|u|^q \in L^t(\mb R^n)$ for some $t>1$ satisfying
 \[\frac{2}{t}+ \frac{\mu}{n}=2.\]
 Thus, for $u \in H^s(\mb R^n)$, by Sobolev Embedding theorems, we must have
 \[\frac{2n-\mu}{n}\leq q \leq \frac{2n-\mu}{n-2s}.\]
 From this, for $u \in X_0$ we have
 \[\left(\int_{\mb R^n} \int_{\mb R^n}\frac{|u(x)|^{2^*_{\mu,s}}|u(y)|^{2^*_{\mu,s}}}{|x-y|^{\mu}}\mathrm{d}x\mathrm{d}y\right)^{\frac{1}{2^*_{\mu,s}}} \leq C(n,\mu)^{\frac{1}{2^*_{\mu,s}}}|u|^2_{2^*_s},\]
 where $C(n,\mu)$ is a suitable constant. We define
 \begin{equation*}
 S^H_s := \inf\limits_{H^s(\mb R^n)\setminus \{0\}} \frac{\int_{\mb R^n}\int_{\mb R^n} \frac{|u(x)-u(y)|^2}{|x-y|^{n+2s}}{\,\mathrm{d}x\mathrm{d}y}}{\left(\int_{\mb R^n} \int_{\mb R^n}\frac{|u(x)|^{2^*_{\mu,s}}|u(y)|^{2^*_{\mu,s}}}{|x-y|^{\mu}}\mathrm{d}x\mathrm{d}y\right)^{\frac{1}{2^*_{\mu,s}}}}
 \end{equation*}
as the best constant which is achieved if and only if $u$ is of the form
\[C\left( \frac{t}{t^2+|x-x_0|^2}\right)^{\frac{n-2s}{2}}, \; \; x \in \mb R^n,\]
for some $x_0 \in \mb R^n$, $C>0$ and $t>0$ (refer Theorem $2.15$ of \cite{chosquas2}). It is well-known that this characterization of $u$ provides the minimizer for $S_s$.  Also,it satisfies
\begin{equation}\label{critical-cho}
(-\De)^s u = \left( \int_{\mb R^n}\frac{|u|^{2^*_{\mu,s}}}{|x-y|^{\mu}}\mathrm{d}y \right)|u|^{2^*_{\mu,s}-2}u \; \; \text{in}\;\; \mb R^n.
\end{equation}
 Moreover,
\begin{equation}\label{relation}
S^H_s = \frac{S_s}{C(n,\mu)^{\frac{1}{2^*_{\mu,s}}}}.
\end{equation}
\noi Consider the family of functions $\{U_{\epsilon}\}$ defined as
\[ U_{\epsilon}(x) = \epsilon^{-\frac{(n-2s)}{2}}\; u^*\left(\frac{x}{\epsilon}\right),\; x \in \mb R^n, \]
where $u^*(x) = \bar{u}\left(\frac{x}{S_s^{1/(2s)}}\right),\; \bar{u}(x) = \frac{\tilde{u}(x)}{\vert u \vert_{2^*_s}}$ and $\tilde{u}(x)= \alpha(\beta^2 + |x|^2)^{-\frac{n-2s}{2}}$ with $\alpha \in \mb R \setminus \{0\}$ and $ \beta >0$ are fixed constants. Then for each $\epsilon > 0$, $U_\epsilon$ satisfies
\[ (-\De)^su = |u|^{2^*_s-2}u \; \;\text{in} \; \mb R^n \]
and verifies the equality
\begin{equation*}
{\int_{\mb R^n} \int_{\mb R^n} \frac{|U_\epsilon(x)-U_{\epsilon}(y)|^2}{|x-y|^{n+2s}}\,\mathrm{d}x\mathrm{d}y =\int_{\mb R^n} |U_\epsilon|^{2^*_s} \,\mathrm{d}x}= {S_s^{\frac{n}{2s}}}.
\end{equation*}
(For a proof, we refer to \cite{bn-serv}.) Then
\[\tilde{U}_\epsilon (x)= S_s^{\frac{(n-\mu)(2s-n)}{4(n+2s-\mu)}}C(n,\mu)^{\frac{2s-n}{2(n+2s-\mu)}}U_\epsilon(x)\]
gives a family of minimizer for $S^H_s$ and satisfies \eqref{critical-cho} and
\[\int_{\mb R^n} \int_{\mb R^n} \frac{|\tilde{U}_\epsilon(x)-\tilde{U}_{\epsilon}(y)|^2}{|x-y|^{n+2s}}\,\mathrm{d}x\mathrm{d}y= \int_{\mb R^n}\int_{\mb R^n} \frac{|\tilde{U}_\epsilon(x)|^{2^*_{\mu,s}}|\tilde{U}_\epsilon(y)|^{2^*_{\mu,s}}}{|x-y|^{\mu}}~\mathrm{d}x\mathrm{d}y = (S^H_s)^{\frac{2n-\mu}{n+2s-\mu}}.\]
Next lemma gives a property about $S^H_s$ which is known to be true for $S_s$.

\begin{Lemma}
Let $n>2s$ and we define
\begin{equation*}
 S^H_s(\Om) := \inf\limits_{X_0\setminus\{0\}} \frac{\int_{Q} \frac{|u(x)-u(y)|^2}{|x-y|^{n+2s}}{\,\mathrm{d}x\mathrm{d}y}}{\left(\int_{\Om} \int_{\Om}\frac{|u(x)|^{2^*_{\mu,s}}|u(y)|^{2^*_{\mu,s}}}{|x-y|^{\mu}}\mathrm{d}x\mathrm{d}y\right)^{\frac{1}{2^*_{\mu,s}}}}.
 \end{equation*}
 Then $S^H_s(\Om)= S^H_s$ and $S^H_s(\Om)$ is never achieved except $\Om = \mb R^n$.
\end{Lemma}
\begin{proof}
Clearly $S^H_s\leq S^H_s(\Om)$. Let $\{u_k\} \subset C^{\infty}_{c}(\mb R^n)$ be a minimizing sequence for $S^H_s$. We choose $\tau_k \in \mb R^n$ and $\theta_k>0$ such that
\[v_k(x) := \tau_k^{\frac{n-2s}{2}} u_k(\tau_k x + \theta_k ) \in C_c^{\infty}(\Om)\]
which satisfies
\[\int_{\mb R^n} \int_{\mb R^n} \frac{|v_k(x)-v_k(y)|^2}{|x-y|^{n+2s}}~\mathrm{d}x\mathrm{d}y =\int_{\mb R^n} \int_{\mb R^n} \frac{|u_k(x)-u_k(y)|^2}{|x-y|^{n+2s}}~\mathrm{d}x\mathrm{d}y \]
and
\[\int_{\Om}\int_{\Om} \frac{|v_k(x)|^{2^*_{\mu,s}}|v_k(y)|^{2^*_{\mu,s}}}{|x-y|^{\mu}}~\mathrm{d}x\mathrm{d}y= \int_{\mb R^n}\int_{\mb R^n} \frac{|u_k(x)|^{2^*_{\mu,s}}|u_k(y)|^{2^*_{\mu,s}}}{|x-y|^{\mu}}~\mathrm{d}x\mathrm{d}y. \]
By definition,
\[S^H_s(\Om) \leq \frac{\int_Q \frac{|v_k(x)-v_k(y)|^2}{|x-y|^{n+2s}}~\mathrm{d}x\mathrm{d}y}{\int_{\Om}\int_{\Om} \frac{|v_k(x)|^{2^*_{\mu,s}}|v_k(y)|^{2^*_{\mu,s}}}{|x-y|^{\mu}}~\mathrm{d}x\mathrm{d}y} \]
which implies $S^H_s(\Om) \leq S^H_s $. Thus, $S^H_s(\Om)$  is never achieved except when $\Om = \mb R^n$ because $\{\tilde{U}_\epsilon\}$ are the only family of minimizers for which the equality holds in Hardy-Littlewood-Sobolev inequality and the best constant is achieved. \hfill{\QED}
\end{proof}

\begin{Definition}
We say that $u \in X_0$ is a weak solution of $(P_\la)$ if
\begin{equation*}
\begin{split}
&\int_{Q}\frac{(u(x)-u(y))(\varphi(x)-\varphi(y))}{|x-y|^{n+2s}}~\mathrm{d}x\mathrm{d}y\\
&\quad = \int_{\Om}\int_{\Om}\frac{|u(x)|^{2^*_{\mu,s}}|u(y)|^{2^*_{\mu,s}-2}u(y)\varphi(y)}{|x-y|^{\mu}}~\mathrm{d}x\mathrm{d}y+ \la \int_{\Om}u\varphi ~dx,
\end{split}
\end{equation*}
for every $\varphi \in C_c^{\infty}(\Om)$.
\end{Definition}
The corresponding energy functional associated to the problem $(P_\la)$ is given by
\[I_{\la}(u)= I(u) = \frac{\|u\|^2}{2} - \frac{1}{22^*_{\mu,s}}\int_{\Om}\int_{\Om}\frac{|u(x)|^{2^*_{\mu,s}}|u(y)|^{2^*_{\mu,s}}}{|x-y|^{\mu}}~\mathrm{d}x\mathrm{d}y
-\frac{\la}{2}\int_{\Om}|u|^2\mathrm{d}x.\]
Using Hardy-Littlewood-Sobolev inequality, we can show that $I \in C^1(X_0,\mb R)$ and
\begin{equation*}
\begin{split}
\langle I^{\prime}(u), \varphi \rangle &= \int_{Q}\frac{(u(x)-u(y))(\varphi(x)-\varphi(y))}{|x-y|^{n+2s}}~\mathrm{d}x\mathrm{d}y\\
&\quad - \int_{\Om}\int_{\Om}\frac{|u(x)|^{2^*_{\mu,s}}|u(y)|^{2^*_{\mu,s}-2}u(y)\varphi(y)}{|x-y|^{\mu}}~\mathrm{d}x\mathrm{d}y- \la \int_{\Om}u\varphi ~dx ,
\end{split}
\end{equation*}
for every $\varphi \in C_c^{\infty}(\Om)$. Thus, $u$ is a weak solution of $(P_\la)$ if and only if $u$ is a critical point of functional $I$. We now state the main results of this paper.

\begin{Theorem}\label{thrm1}
Let $\la_{1}$ denote the first eigenvalue of $(-\De)^s$ with homogenous Dirichlet boundary condition in $\mb R^n \setminus \Om$. Then, for any $\la \in (0,\la_1)$, if $n \geq 4s$ for $s \in (0,1)$, $(P_\la)$ has a nontrivial solution.
\end{Theorem}

\begin{Theorem}\label{newthrm}
Let $s \in (0,1)$ and $ 2s<n<4s $, then there exist $\bar{\la}>0$ such that for any $\la > \bar \la $ different from the eigenvalues of $(-\De)^s$ with homogenous Dirichlet boundary condition in $\mb R^n \setminus \Om$, $(P_\la)$ has a nontrivial solution.
\end{Theorem}

\begin{Theorem}\label{thrm2}
Assume $n>2s$ and $s\in (0,1)$, then there exists a constant $\la_*$ such that if there are $q$ number of eigenvalues lying between $\la $ and $\la + \la_*$, then $(P_\la)$ has $q$ distinct pairs of solutions.
\end{Theorem}

\begin{Theorem}\label{thrm4}
Let $0 \leq u \in X_0$, $n>2s$ and $\la >0$ be such that
\begin{equation*}
\begin{split}
&\int_{Q}\frac{(u(x)-u(y))(\varphi(x)-\varphi(y))}{|x-y|^{n+2s}}~\mathrm{d}x\mathrm{d}y\\
&\quad = \int_{\Om}\int_{\Om}\frac{|u(x)|^{2^*_{\mu,s}}|u(y)|^{2^*_{\mu,s}-2}u(y)\varphi(y)}{|x-y|^{\mu}}~\mathrm{d}x\mathrm{d}y+ \la \int_{\Om}u\varphi ~dx,
\end{split}
\end{equation*}
for every $\varphi \in C_c^{\infty}(\Om)$, i.e. $u$ is a nonnegative weak solution of $(P_\la)$. Then, $u \in L^{\infty}(\Om)$.
\end{Theorem}


\begin{Theorem}\label{thrm3}
Let $n>2s$, $\la <0$ and $\Om \neq \mb R^n$ be a strictly star shaped (with respect to origin), $C^{1,1}$ and bounded domain in $\mb R^n$, then  $(P_\la)$ cannot have a  nonnegative nontrivial solution.
\end{Theorem}

\section{Preliminary Results}
We consider $\Om$ to be a bounded domain in $\mb R^n$ with Lipschitz boundary and $\la$ to be a real parameter throughout this paper.

\begin{Definition}
Let $I$ be a $C^1$ functional defined on Banach space $X$, we say that $\{v_k\}$ is a Palais-Smale sequence of $I$ at $c$ (denoted by $(PS)_c$) if
\[I(v_k) \rightarrow c, \; \text{ and } \; I^{\prime}(v_k) \rightarrow 0, \; \text{as} \; k \rightarrow +\infty.\]
And we say that $I$ satisfies the Palais-Smale condition at the level $c$, if every Palais-Smale sequence at $c$ has a convergent subsequence.
\end{Definition}
The following lemmas can be proved using the standard methods but we give some of their proof here for the sake of completeness. To begin, we recall that pointwise convergence of a bounded sequence implies weak convergence.

\begin{Lemma}
Let $q \in (1, \infty)$ and $\{u_k\}$ be a bounded sequence in $L^q(\mb R^n)$. If $u_k \rightarrow v$ almost everywhere in $\mb R^n$ as $k \rightarrow \infty$, then $u_k \rightharpoonup u $ weakly in $L^q(\mb R^n)$.
\end{Lemma}

\begin{Lemma}\label{conv}
Let $n>2s$, $0<\mu<n$ and $\{u_k\}$ be a bounded sequence in $L^{2^*_s}(\mb R^n)$ such that $u_k \rightarrow u$ almost everywhere in $\mb R^n$ as $n \rightarrow \infty$, then the following hold,
\begin{equation*}
\begin{split}
\int_{\mb R^n} \int_{\mb R^n} \frac{|u_k(x)|^{2^*_{\mu,s}}|u_k(y)|^{2^*_{\mu,s}}}{|x-y|^{\mu}}~\mathrm{d}x\mathrm{d}y &- \int_{\mb R^n} \int_{\mb R^n} \frac{|(u_k-u)(x)|^{2^*_{\mu,s}}|(u_k-u)(y)|^{2^*_{\mu,s}}}{|x-y|^{\mu}}~\mathrm{d}x\mathrm{d}y\\
& \rightarrow \int_{\mb R^n} \int_{\mb R^n} \frac{|u(x)|^{2^*_{\mu,s}}|u(y)|^{2^*_{\mu,s}}}{|x-y|^{\mu}}~\mathrm{d}x \mathrm{d}y \; \text{ as}\;  k \rightarrow \infty.
\end{split}
\end{equation*}
\end{Lemma}
\begin{proof}
Proof follows similarly as proof of lemma 2.3 \cite{myang}.\hfill{\QED}
\end{proof}

\begin{Lemma}\label{PSbdd}
Let $n>2s$, $0<\mu<n$. Then every Palais-Smale sequence of $I$ is bounded and its weak limit is a weak solution of $(P_\la)$.
\end{Lemma}
\begin{proof}
Let $\{u_k\}$ be a Palais-Smale sequence of $I$ at $c \in \mb R^n$. We can assume $c\geq 0$ and by definition, there exist positive constants $C_1$ and $C_2$ such that
\[|I(u_k)| \leq C_1,\; \text{ and } \; |\langle I^{\prime}(u_k),\frac{u_k}{\|u_k\|} \rangle|\leq C_2.\]
We have
\begin{equation*}
\begin{split}
\frac12 \langle I^{\prime}(u_k),u_k \rangle = I(u_k)- \frac{n+2s-\mu}{2(2n-\mu)}\int_{\Om}\int_{\Om}\frac{|u_k(x)|^{2^*_{\mu,s}}|u_k(y)|^{2^*_{\mu,s}}}{|x-y|^{\mu}}~\mathrm{d}x\mathrm{d}y
\end{split}
\end{equation*}
which implies
\[\int_{\Om}\int_{\Om}\frac{|u_k(x)|^{2^*_{\mu,s}}|u_k(y)|^{2^*_{\mu,s}}}{|x-y|^{\mu}}~\mathrm{d}x\mathrm{d}y \leq C_2(1+\|u_k\|),\]
for some positive constant $C_2$. Also, we have
\begin{equation*}
\begin{split}
I(u_k)+ \frac12 \langle I^{\prime}(u_k),u_k \rangle &= \|u_k\|^2- \frac{3n-2s-\mu}{2(2n-\mu)}\int_{\Om}\int_{\Om}\frac{|u_k(x)|^{2^*_{\mu,s}}|u_k(y)|^{2^*_{\mu,s}}}{|x-y|^{\mu}}~\mathrm{d}x\mathrm{d}y\\
& \leq C_3(1+\|u_k\|),
\end{split}
\end{equation*}
for some positive constant $C_3$. This implies
\begin{equation*}
\begin{split}
\|u_k\|^2 & \leq \frac{3n-2s-\mu}{2(2n-\mu)}\int_{\Om}\int_{\Om}\frac{|u_k(x)|^{2^*_{\mu,s}}|u_k(y)|^{2^*_{\mu,s}}}{|x-y|^{\mu}}~\mathrm{d}x\mathrm{d}y
+C_3(1+\|u_k\|)\\
& \leq C_4(1+\|u_k\|),
\end{split}
\end{equation*}
for some positive constant $C_4$. Thus, we get $\{u_k\}$ to be a bounded sequence in $X_0$ which implies that there exist a subsequence and $u \in X_0$, still denoted by $u_k$. such that $u_k \rightharpoonup u$ in $X_0$ and also $u_k \rightharpoonup u$ in $L^{2^*_s}(\Om)$ as $k \rightarrow +\infty$. Then
\[|u_k|^{2^*_{\mu,s}} \rightharpoonup |u|^{2^*_{\mu,s}} \; \text{ in } L^{\frac{2n}{2n-\mu}}(\Om)\]
and
\[|u_k|^{2^*_{\mu,s}-2}u_k \rightharpoonup |u|^{2^*_{\mu,s}-2}u \; \text{ in } L^{\frac{2n}{n+2s-\mu}}(\Om)\]
as $k \rightarrow +\infty$.  The Reisz potential defines a continuous map from $L^{\frac{2n}{2n-\mu}}(\Om)$ to $L^{\frac{2n}{\mu}}(\Om)$, using Hardy-Littlewood-Sobolev inequality. This gives
\[\int_{\Om} \frac{|u_k(y)|^{2^*_{\mu,s}}}{|x-y|^{\mu}}~\mathrm{d}y \rightharpoonup \int_{\Om} \frac{|u(y)|^{2^*_{\mu,s}}}{|x-y|^{\mu}}~\mathrm{d}y \; \text{ in } L^{\frac{2n}{\mu}}(\Om)\]
as $k \rightarrow +\infty$. Combining all these, we get
\[\int_{\Om} \frac{|u_k(y)|^{2^*_{\mu,s}}|u_k(x)|^{2^*_{\mu,s}-2}u_k(x)}{|x-y|^{\mu}}~\mathrm{d}y \rightharpoonup \int_{\Om} \frac{|u(y)|^{2^*_{\mu,s}}|u(x)|^{2^*_{\mu,s}-2}u(x)}{|x-y|^{\mu}}~\mathrm{d}y \; \text{ in } L^{\frac{2n}{n+2s}}(\Om)\]
as $k \rightarrow +\infty$. Since $I^{\prime}(u_k) \rightarrow 0$ as $k \rightarrow +\infty$, for any $\varphi \in C_c^{\infty}(\Om)$, we get
\begin{equation*}
\begin{split}
&\lim\limits_{k\rightarrow +\infty} \left(\int_{Q}\frac{(u_k(x)-u_k(y))(\varphi(x)-\varphi(y))}{|x-y|^{n+2s}}~\mathrm{d}x\mathrm{d}y \right.\\
&\quad - \left.\int_{\Om}\int_{\Om}\frac{|u_k(x)|^{2^*_{\mu,s}}|u_k(y)|^{2^*_{\mu,s}-2}u_k(y)\varphi(y)}{|x-y|^{\mu}}~\mathrm{d}x\mathrm{d}y- \la \int_{\Om}u_k\varphi ~dx\right) = 0.
\end{split}
\end{equation*}
This gives
\begin{equation*}
\begin{split}
0=&\int_{Q}\frac{(u(x)-u(y))(\varphi(x)-\varphi(y))}{|x-y|^{n+2s}}~\mathrm{d}x\mathrm{d}y\\
& \quad -\int_{\Om}\int_{\Om}\frac{|u(x)|^{2^*_{\mu,s}}|u(y)|^{2^*_{\mu,s}-2}u(y)\varphi(y)}{|x-y|^{\mu}}~\mathrm{d}x\mathrm{d}y- \la \int_{\Om}u\varphi ~dx
\end{split}
\end{equation*}
for any $\varphi \in C_c^{\infty}(\infty)$. Thus, $u$ is a weak solution of $(P_\la)$.\hfill{\QED}
\end{proof}

\noi Let $u$ be the solution obtained in above lemma and we take $\varphi = u$ as the test function in $(P_\la)$, then we get
\[\|u\|^2 = \int_{\Om}\int_{\Om}\frac{|u(x)|^{2^*_{\mu,s}}|u(y)|^{2^*_{\mu,s}}}{|x-y|^{\mu}}~\mathrm{d}x\mathrm{d}y+ \la \int_{\Om}u^2~\mathrm{d}x.\]
So,
\begin{equation}\label{I-pos}
I(u)= \frac{n+2s-\mu}{2(2n-\mu)} \int_{\Om}\int_{\Om}\frac{|u(x)|^{2^*_{\mu,s}}|u(y)|^{2^*_{\mu,s}}}{|x-y|^{\mu}}~\mathrm{d}x\mathrm{d}y \geq 0.
\end{equation}

 \begin{Lemma}\label{PSlevel}
 Let $n>2s$, $0<\mu<n$ and $\{u_k\}$ be a $(PS)_c$ sequence of $I$ with
 \[c < \frac{n+2s-\mu}{2(2n-\mu)}(S^H_s)^{\frac{2n-\mu}{n+2s-\mu}}.\]
 Then $\{u_k\}$ has a convergent subsequence.
 \end{Lemma}
\begin{proof}
Let $u$ be the weak limit of $\{u_k\}$ obtained using lemma \ref{PSbdd}. We set $w_k := u_k-u $, then $w_k \rightharpoonup 0$ in $X_0$ and $w_k \rightarrow 0$ a.e. in $\Om$ as $k \rightarrow +\infty$. By Brezis-Lieb Lemma, we have
\[\|u_k\|^2 = \|w_k\|^2 +\|u\|^2+ o_k(1),\;\; \text{and} \;\; |u_k|^2_2 = |w_k|^2_2+ |u|^2_2 + o_k(1).\]
Also, using Lemma \ref{conv}, we have
\begin{equation*}
\begin{split}
&\int_{\Om}\int_{\Om}\frac{|u_k(x)|^{2^*_{\mu,s}}|u_k(y)|^{2^*_{\mu,s}}}{|x-y|^{\mu}}~\mathrm{d}x\mathrm{d}y \\
& =  \int_{\Om}\int_{\Om}\frac{|w_k(x)|^{2^*_{\mu,s}}|w_k(y)|^{2^*_{\mu,s}}}{|x-y|^{\mu}}~\mathrm{d}x\mathrm{d}y +
\int_{\Om}\int_{\Om}\frac{|u(x)|^{2^*_{\mu,s}}|u(y)|^{2^*_{\mu,s}}}{|x-y|^{\mu}}~\mathrm{d}x\mathrm{d}y+o_k(1)
\end{split}
\end{equation*}
Since $I(u_k) \rightarrow c $ as $k \rightarrow +\infty$, we get

\begin{equation*}
\begin{split}
c =\lim\limits_{k \rightarrow +\infty} I(u_k)= \lim\limits_{k \rightarrow +\infty}\left(\frac{\|u_k\|^2}{2} - \frac{1}{22^*_{\mu,s}}\int_{\Om}\int_{\Om}\frac{|u_k(x)|^{2^*_{\mu,s}}|u_k(y)|^{2^*_{\mu,s}}}{|x-y|^{\mu}}~\mathrm{d}x\mathrm{d}y
-\frac{\la}{2}\int_{\Om}|u_k|^2\mathrm{d}x\right)
\end{split}
\end{equation*}
\begin{align}\label{PSlevel1}
& = \frac{\|w_k\|^2}{2}- \frac{\la}{2}\int_{\Om}w_k^2~\mathrm{d}x +\frac{\|u\|^2}{2}- \frac{\la}{2}\int_{\Om}u^2\mathrm{d}x\notag\\
& \quad - \frac{1}{22^*_{\mu,s}}\int_{\Om}\int_{\Om}\frac{|w_k(x)|^{2^*_{\mu,s}}|w_k(y)|^{2^*_{\mu,s}}}{|x-y|^{\mu}}~\mathrm{d}x\mathrm{d}y
-\frac{1}{22^*_{\mu,s}}\int_{\Om}\int_{\Om}\frac{|u(x)|^{2^*_{\mu,s}}|u(y)|^{2^*_{\mu,s}}}{|x-y|^{\mu}}~\mathrm{d}x\mathrm{d}y+o_k(1)\notag\\
& = I(u)+ \frac{\|w_k\|^2}{2}- \frac{\la}{2}\int_{\Om}w_k^2~\mathrm{d}x - \frac{1}{22^*_{\mu,s}}\int_{\Om}\int_{\Om}\frac{|w_k(x)|^{2^*_{\mu,s}}|w_k(y)|^{2^*_{\mu,s}}}{|x-y|^{\mu}}~\mathrm{d}x\mathrm{d}y + o_k(1)\notag\\
& \geq \frac{\|w_k\|^2}{2}- \frac{1}{22^*_{\mu,s}}\int_{\Om}\int_{\Om}\frac{|w_k(x)|^{2^*_{\mu,s}}|w_k(y)|^{2^*_{\mu,s}}}{|x-y|^{\mu}}~\mathrm{d}x\mathrm{d}y + o_k(1),
\end{align}
using \eqref{I-pos} and $\int_{\Om}w_k^2~\mathrm{d}x \rightarrow 0$ as $k \rightarrow +\infty$ (because $X_0\hookrightarrow L^{2}(\Om)$ compactly). In a similar manner, since $u$ is a weak solution of $(P_\la)$, $u$ must be a critical point of $I$ which gives $\langle I^{\prime}(u),u\rangle =0$ that is
\begin{align}\label{PSlevel2}
o_k(1)&= \|u_k\|^2 -\int_{\Om}\int_{\Om}\frac{|u_k(x)|^{2^*_{\mu,s}}|u_k(y)|^{2^*_{\mu,s}}}{|x-y|^{\mu}}~\mathrm{d}x\mathrm{d}y
-{\la}\int_{\Om}|u_k|^2\mathrm{d}x\notag\\
& = \|w_k\|^2-{\la}\int_{\Om}|w_k|^2\mathrm{d}x+ \|u\|^2-{\la}\int_{\Om}|u|^2\mathrm{d}x\notag\\
&\quad-\int_{\Om}\int_{\Om}\frac{|w_k(x)|^{2^*_{\mu,s}}|w_k(y)|^{2^*_{\mu,s}}}{|x-y|^{\mu}}~\mathrm{d}x\mathrm{d}y-
\int_{\Om}\int_{\Om}\frac{|u(x)|^{2^*_{\mu,s}}|u(y)|^{2^*_{\mu,s}}}{|x-y|^{\mu}}~\mathrm{d}x\mathrm{d}y+ o_k(1)\notag\\
& =\langle I^\prime(u),u\rangle + \|w_k\|^2-{\la}\int_{\Om}|w_k|^2\mathrm{d}x -\int_{\Om}\int_{\Om}\frac{|w_k(x)|^{2^*_{\mu,s}}|w_k(y)|^{2^*_{\mu,s}}}{|x-y|^{\mu}}~\mathrm{d}x\mathrm{d}y +o_k(1)\notag \\
& = \|w_k\|^2 -\int_{\Om}\int_{\Om}\frac{|w_k(x)|^{2^*_{\mu,s}}|w_k(y)|^{2^*_{\mu,s}}}{|x-y|^{\mu}}~\mathrm{d}x\mathrm{d}y +o_k(1).
\end{align}
This implies
\[\lim\limits_{k \rightarrow +\infty}\|w_k\|^2 = \lim\limits_{k \rightarrow +\infty}\int_{\Om}\int_{\Om}\frac{|w_k(x)|^{2^*_{\mu,s}}|w_k(y)|^{2^*_{\mu,s}}}{|x-y|^{\mu}}~\mathrm{d}x\mathrm{d}y = a, \]
where $a$ is nonnegative constant. From \eqref{PSlevel1} and \eqref{PSlevel2}, we deduce
\[c \geq \frac{n+2s-\mu}{2(2n-\mu)}a. \]
Using definition of $S^H_s$, we get
\[S^H_s \left(\int_{\Om}\int_{\Om}\frac{|w_k(x)|^{2^*_{\mu,s}}|w_k(y)|^{2^*_{\mu,s}}}{|x-y|^{\mu}}~\mathrm{d}x\mathrm{d}y  \right)^{\frac{n-2s}{2n-\mu}} \leq \|w_k\|^2,\]
which gives $a \geq S^H_s a^{\frac{n-2s}{2n-\mu}}$. Thus, either $a=0$ or $a \geq (S^H_s)^{\frac{2n-\mu}{n+2s-\mu}}$. If $a=0$, we are done, else $a \geq (S^H_s)^{\frac{2n-\mu}{n+2s-\mu}}$ gives
\[\frac{n+2s-\mu}{2(2n-\mu)}(S^H_s)^{\frac{2n-\mu}{n+2s-\mu}} \leq c.\]
This contradicts the hypothesis that
\[c < \frac{n+2s-\mu}{2(2n-\mu)}(S^H_s)^{\frac{2n-\mu}{n+2s-\mu}}.\]
Thus, $a=0$ which implies $\|u_k-u\| \rightarrow 0$ as $k \rightarrow +\infty$.\hfill{\QED}
\end{proof}

\section{Proof of Theorem \ref{thrm1}}
We fix $n \geq 4s$ and $\Om$ be a smooth bounded domain in $\mb R^n$. We divide the proof of \ref{thrm1}  considering two cases.
\subsection{Case (1):  $\la \in (0,\la_1)$}
Without loss of generality, we assume $0 \in \Om$ and fix $\delta>0$ such that $B_{4\delta}\subset \Om$. Let $\eta \in C^{\infty}(\mb R^n)$ be such that $0\leq \eta \leq 1$ in $\mb R^n$, $\eta \equiv 1$ in $B_{\delta}$ and $\eta \equiv 0$ in $\mb R^n \setminus B_{2\delta}$. For $\epsilon >0$, we denote by $u_{\epsilon}$ the following function
\[u_\epsilon(x) = \eta(x)U_{\epsilon}(x),\]
for $x \in \mb R^n$, where $U_\epsilon$ is defined in section 2. We have the following results for $u_\epsilon$ using Proposition $21$ and $22$ of \cite{bn-serv}.
\begin{Proposition}\label{estimates}
Let $s\in(0,1)$ and $n>2s$. Then, the following estimates holds true as $\epsilon \rightarrow 0 $ \\
\begin{enumerate}
\item[(i)] $\ds \int_{\mb R^n}\frac{|u_\epsilon(x)- u_\epsilon(y)|^2}{|x-y|^{n+2s}}~\mathrm{d}x\mathrm{d}y = S_s^{n/(2s)}+o(\epsilon^{n-2s})$,\\
\item[(ii)] $\ds \int_{\Om}|u_\epsilon|^{2^*_s}~\mathrm{d}x = S_s^{n/(2s)}+ o(\epsilon^n)$,\\
\item[(iii)] $$\int_{\Om}|u_{\epsilon}(x)|^2~\mathrm{d}x=
\left\{
	\begin{array}{ll}
		C_s\epsilon^{2s}+o(\epsilon^{n-2s}) & \mbox{if } n>4s \\
		C_s\epsilon^{2s}|\log \epsilon|+o(\epsilon^{2s}) & \mbox{if } n=4s \\
        C_s\epsilon^{n-2s}+o(\epsilon^{2s}) & \mbox{if } n<4s
	\end{array}
\right.,$$
\end{enumerate}
for some positive constant $C_s$, depending on $s$.
\end{Proposition}
\noi Using \eqref{relation}, Proposition \ref{estimates}(i) can be written as
\begin{equation}\label{esti-new}
\int_{\mb R^n}\frac{|u_\epsilon(x)- u_\epsilon(y)|^2}{|x-y|^{n+2s}}~\mathrm{d}x\mathrm{d}y \leq S_s^{n/(2s)}+o(\epsilon^{n-2s})=(C(n,\mu))^{\frac{n-2s}{2n-\mu}}(S^H_s)^{\frac{n}{2s}}+o(\epsilon^{n-2s}).
\end{equation}
We now prove the following proposition in the spirit of section 3 of \cite{myang}.
\begin{Proposition}\label{estimates1}
The following estimates holds true:
\begin{equation*}
\left(\int_{\Om}\int_{\Om}\frac{|u_\epsilon(x)|^{2^*_{\mu,s}}|u_\epsilon(y)|^{2^*_{\mu,s}}}{|x-y|^{\mu}}~\mathrm{d}x\mathrm{d}y \right)^{\frac{n-2s}{2n-\mu}}\leq (C(n,\mu))^{\frac{n(n-2s)}{2s(2n-\mu)}} (S^H_s)^{\frac{n-2s}{2}}+ o(\epsilon^{n-2s}),
\end{equation*}
and
\begin{equation*}
\left(\int_{\Om}\int_{\Om}\frac{|u_\epsilon(x)|^{2^*_{\mu,s}}|u_\epsilon(y)|^{2^*_{\mu,s}}}{|x-y|^{\mu}}~\mathrm{d}x\mathrm{d}y \right)^{\frac{n-2s}{2n-\mu}}\geq \left((C(n,\mu))^{\frac{n}{2s}} (S^H_s)^{\frac{2n-\mu}{2}}+ o(\epsilon^{n})\right)^{\frac{n-2s}{2n-\mu}}.
\end{equation*}
\end{Proposition}
\begin{proof}
By Hardy-Littlewood-Sobolev inequality, Proposition \ref{estimates}(ii) and \ref{relation}, we get
\begin{equation*}
\begin{split}
&\left(\int_{\Om}\int_{\Om}\frac{|u_\epsilon(x)|^{2^*_{\mu,s}}|u_\epsilon(y)|^{2^*_{\mu,s}}}{|x-y|^{\mu}}~\mathrm{d}x\mathrm{d}y \right)^{\frac{n-2s}{2n-\mu}}\\
&\leq (C(n,\mu))^{\frac{n-2s}{2n-\mu}} |u_\epsilon|^2_{2^*_s} =(C(n,\mu))^{\frac{n-2s}{2n-\mu}}\left(S_s^{n/(2s)}+ o(\epsilon^n)\right)^{\frac{n-2s}{n}}\\
&= (C(n,\mu))^{\frac{n-2s}{2n-\mu}}\left( (C(n,\mu))^{\frac{n(n-2s)}{2s(2n-\mu)}}(S^H_s)^{\frac{n}{2s}}+ o(\epsilon^n)\right)^{\frac{n-2s}{n}}\\
&=(C(n,\mu))^{\frac{n(n-2s)}{2s(2n-\mu)}} (S^H_s)^{\frac{n-2s}{2s}}+ o(\epsilon^{n-2s}).
\end{split}
\end{equation*}
Next, we consider
\begin{equation}\label{har-esti3}
\begin{split}
&\int_{\Om}\int_{\Om}\frac{|u_\epsilon(x)|^{2^*_{\mu,s}}|u_\epsilon(y)|^{2^*_{\mu,s}}}{|x-y|^{\mu}}~\mathrm{d}x\mathrm{d}y \\
 &\geq \int_{B_{\delta}}\int_{B_{\delta}}\frac{|u_\epsilon(x)|^{2^*_{\mu,s}}|u_\epsilon(y)|^{2^*_{\mu,s}}}{|x-y|^{\mu}}~\mathrm{d}x\mathrm{d}y
 =\int_{B_{\delta}}\int_{B_{\delta}}\frac{|U_\epsilon(x)|^{2^*_{\mu,s}}|U_\epsilon(y)|^{2^*_{\mu,s}}}{|x-y|^{\mu}}~\mathrm{d}x\mathrm{d}y\\
& = \int_{\mb R^n}\int_{\mb R^n}\frac{|U_\epsilon(x)|^{2^*_{\mu,s}}|U_\epsilon(y)|^{2^*_{\mu,s}}}{|x-y|^{\mu}}~\mathrm{d}x\mathrm{d}y -2 \int_{\mb R^n \setminus B_{\delta}}\int_{B_{\delta}}\frac{|U_\epsilon(x)|^{2^*_{\mu,s}}|U_\epsilon(y)|^{2^*_{\mu,s}}}{|x-y|^{\mu}}~\mathrm{d}x\mathrm{d}y\\
& \quad \quad- \int_{\mb R^n \setminus B_{\delta}}\int_{\mb R^n \setminus B_{\delta}}\frac{|U_\epsilon(x)|^{2^*_{\mu,s}}|U_\epsilon(y)|^{2^*_{\mu,s}}}{|x-y|^{\mu}}~\mathrm{d}x\mathrm{d}y.
\end{split}
\end{equation}
We estimate the integrals in R.H.S. of \eqref{har-esti3} separately. Firstly, consider
\begin{equation}\label{har-esti4}
\int_{\mb R^n}\int_{\mb R^n}\frac{|U_\epsilon(x)|^{2^*_{\mu,s}}|U_\epsilon(y)|^{2^*_{\mu,s}}}{|x-y|^{\mu}}~\mathrm{d}x\mathrm{d}y
= \left( \frac{\|U_\epsilon\|^2}{S^H_s}\right)^{\frac{2n-\mu}{n-2s}}= (C(n,\mu))^{\frac{n}{(2s)}}(S^H_s)^{\frac{2n-\mu}{2s}}.
\end{equation}
Secondly, consider
\begin{equation*}
\begin{split}
&\int_{\mb R^n \setminus B_{\delta}}\int_{B_{\delta}}\frac{|U_\epsilon(x)|^{2^*_{\mu,s}}|U_\epsilon(y)|^{2^*_{\mu,s}}}{|x-y|^{\mu}}~\mathrm{d}x\mathrm{d}y\\
&\leq C_{1,s}\int_{\mb R^n \setminus B_{\delta}}\int_{B_{\delta}} \frac{\epsilon^{\mu-2n}}{|x-y|^{\mu}\left( 1+ |\frac{x}{\epsilon}|^2 \right)^{\frac{2n-\mu}{2}}\left( 1+ |\frac{y}{\epsilon}|^2 \right)^{\frac{2n-\mu}{2}}}~\mathrm{d}x\mathrm{d}y\\
& = \epsilon^{2n-\mu}C_{2,s}\int_{\mb R^n \setminus B_{\delta}}\int_{B_{\delta}} \frac{1}{|x-y|^{\mu}\left( \epsilon^2+ |x|^2 \right)^{\frac{2n-\mu}{2}}\left( \epsilon^2+ |y|^2 \right)^{\frac{2n-\mu}{2}}}~\mathrm{d}x\mathrm{d}y\\
& \leq \epsilon^{2n-\mu}C_{2,s} \int_{\mb R^n \setminus B_{\delta}} \frac{1}{|x|^{2n-\mu}(|x|-\delta)^\mu}\mathrm{d}x \int_{ B_{\delta}} \frac{1}{(\epsilon^2+|y|^2)^{\frac{2n-\mu}{2}}}\mathrm{d}y
\end{split}
\end{equation*}
\begin{equation}\label{har-esti5}
\begin{split}
& = o(\epsilon^n) \int_{0}^{\delta/\epsilon}\frac{t^{n-1}}{(1+t^2)^{\frac{2n-\mu}{2}}}~\mathrm{d}t \leq o(\epsilon^n) \int_{0}^{+\infty}\frac{t^{n-1}}{(1+t^2)^{\frac{2n-\mu}{2}}}~\mathrm{d}t= o(\epsilon^n),
\end{split}
\end{equation}
where $C_{1,s}, C_{2,s}$ are appropriate positive constants. Lastly, in a similar manner we have

\begin{equation*}
\begin{split}
&\int_{\mb R^n \setminus B_{\delta}}\int_{\mb R^n \setminus B_{\delta}}\frac{|U_\epsilon(x)|^{2^*_{\mu,s}}|U_\epsilon(y)|^{2^*_{\mu,s}}}{|x-y|^{\mu}}~\mathrm{d}x\mathrm{d}y\\
&\leq C_{1,s}\int_{\mb R^n \setminus B_{\delta}}\int_{\mb R^n\setminus B_{\delta}} \frac{\epsilon^{\mu-2n}}{|x-y|^{\mu}\left( 1+ |\frac{x}{\epsilon}|^2 \right)^{\frac{2n-\mu}{2}}\left( 1+ |\frac{y}{\epsilon}|^2 \right)^{\frac{2n-\mu}{2}}}~\mathrm{d}x\mathrm{d}y
\end{split}
\end{equation*}
\begin{equation}\label{har-esti6}
\begin{split}
& = \epsilon^{2n-\mu}C_{2,s}\int_{\mb R^n \setminus B_{\delta}}\int_{\mb R^n \setminus B_{\delta}} \frac{1}{|x-y|^{\mu}\left( \epsilon^2+ |x|^2 \right)^{\frac{2n-\mu}{2}}\left( \epsilon^2+ |y|^2 \right)^{\frac{2n-\mu}{2}}}~\mathrm{d}x\mathrm{d}y\\
& \leq \epsilon^{2n-\mu}C_{2,s} \int_{\mb R^n \setminus B_{\delta}} \int_{ B_{\delta}} \frac{1}{|x-y|^\mu|x|^{2n-\mu}|y|^{2n-\mu}}\mathrm{d}x\mathrm{d}y = o(\epsilon^{2n-\mu}).
\end{split}
\end{equation}
Using  \eqref{har-esti4}, \eqref{har-esti5} and \eqref{har-esti6} in \eqref{har-esti3}, we get
 \begin{equation}\label{har-esti7}
\left( \int_{\Om}\int_{\Om}\frac{|u_\epsilon(x)|^{2^*_{\mu,s}}|u_\epsilon(y)|^{2^*_{\mu,s}}}{|x-y|^{\mu}}~\mathrm{d}x\mathrm{d}y \right)^{\frac{n-2s}{2n-\mu}} \geq \left( (C(n,\mu))^{\frac{n}{2s}}(S^H_s)^{\frac{2n-\mu}{2s}}- o(\epsilon^n)\right)^{\frac{n-2s}{2n-\mu}}.
 \end{equation}
 This completes the proof.\hfill{\QED}
\end{proof}
\begin{Remark}\label{rmrk}
\eqref{har-esti7} and \eqref{esti-new} still holds when {$2s<n<4s$}.
\end{Remark}
We prove the existence of solution to $(P_\la)$ using an invariant of mountain pass lemma.

\begin{Lemma}\label{MGP}
If $n>2s$ and $\la \in (0,\la_1)$, then the energy functional $I$ satisfies the following properties:
\begin{enumerate}
\item[(i)]there exist $\beta,\rho>0$ such that $I(u) \geq \beta$ when $\|u\|=\rho$,
\item[(ii)] there exist $\tilde{u} \in X_0$ such that $\|\tilde{u}\| > \rho$ and $I(\tilde{u}) <0$.
\end{enumerate}
\end{Lemma}
\begin{proof}\begin{enumerate}
\item[(i)] Since $\la \in (0,\la_1)$, using Sobolev embedding and Hardy-Littlewood-Sobolev inequality, we get
\begin{equation*}
\begin{split}
I(u)& \geq \frac12 \left(1-\frac{\la}{\la_1}\right)\|u\|^2 - \frac{1}{2 2^{*}_{\mu,s}}C_1|u|^{\frac{2(2n-\mu)}{n-2s}}\\
&\geq \frac12 \left(1-\frac{\la}{\la_1}\right)\|u\|^2 - \frac{1}{2 2^{*}_{\mu,s}}C_1C_2 \|u\|^{\frac{2(2n-\mu)}{n-2s}},
\end{split}
\end{equation*}
for all $u \in X_0 \setminus \{0\}$, where $C_1,C_2$ are positive constants. Since $0<\mu<n$, so $2< 2\left(\frac{2n-\mu}{n-2s}\right)$. Thus, some $\beta, \rho>0$ can be chosen such that $I(u)\geq \beta$ when $\|u\|= \rho$.
\item[(ii)] Fix $u_0 \in X_0\setminus \{0\}$, since $I(tu_0) \rightarrow -\infty$ as $t \rightarrow \infty$ we get
\[I(tu_0)=  \frac{t^2\|u_0\|^2}{2} - \frac{t^{22^*_{\mu,s}}}{22^*_{\mu,s}}\int_{\Om}\int_{\Om}\frac{|u_0(x)|^{2^*_{\mu,s}}|u_0(y)|^{2^*_{\mu,s}}}{|x-y|^{\mu}}~\mathrm{d}x\mathrm{d}y
-\frac{\la t^2}{2}\int_{\Om}|u_0|^2\mathrm{d}x <0\]
for sufficiently large $t>0$. This implies, we can obtain $\tilde{u}= t_0u_0 \in X_0$ for some $t_0>0$ such that $\|\tilde{u}\| > \rho$ and $I(\tilde{u}) <0$.\hfill{\QED}
\end{enumerate}
\end{proof}
\begin{Proposition}\cite{myang}\label{MP}
Using lemma \ref{MGP} and the mountain pass lemma without $(PS)$ condition \cite{Willem}, there exists a $(PS)$ sequence $\{u_k\}$ such that $I(u_k) \rightarrow c$ and $I^{\prime}(u_k) \rightarrow 0$ in $X_0^*$ (dual of $X_0$) at the minimax level
\[c^* = \inf_{\gamma \in \Gamma} \max_{t \in [0,1]} I(\gamma(t)) >0,\]
where
\[\Gamma :=  \{\gamma \in C([0,1],X_0): \; \gamma(0)=0, I(\gamma(1))<0\}.\]
\end{Proposition}


\noi \textbf{Proof of Theorem \ref{thrm1}: ($n\geq 4s$, $\la \in (0,\la_1)$)}\\
 Before proving this theorem, we claim that there exist $w \in X_0\setminus \{0\}$ such that
\begin{equation}\label{elem1}
\frac{\|w\|^2 - \la \int_{\Om}|w|^2\mathrm{d}x}{\left(\int_{\Om}\int_{\Om}\frac{|w(x)|^{2^*_{\mu,s}}|w(y)|^{2^*_{\mu,s}}}{|x-y|^{\mu}}~\mathrm{d}x\mathrm{d}y \right)^{\frac{n-2s}{2n-\mu}}} < S^H_s
\end{equation}
If $n=4s$, using Proposition \ref{estimates}(iii), \eqref{esti-new} and \eqref{har-esti7}, we get
\begin{equation}\label{elem2}
\begin{split}
\frac{\|u_\epsilon\|^2 - \la \int_{\Om}|u_\epsilon|^2\mathrm{d}x}{\left(\int_{\Om}\int_{\Om}\frac{|u_\epsilon(x)|^{2^*_{\mu,s}}|u_\epsilon(y)|^{2^*_{\mu,s}}}{|x-y|^{\mu}}~\mathrm{d}x\mathrm{d}y \right)^{\frac{2s}{8s-\mu}}} & \leq \frac{(C(4s,\mu))^{\frac{4s}{8s-\mu}}(S^H_s)^{2s}- \la C_s \epsilon^{2s}|\log \epsilon| + o(\epsilon^{2s})}{\left((C(4s,\mu))^2(S^H_s)^{\frac{8s-\mu}{2s}}- o(\epsilon^{4s}) \right)^{\frac{2s}{8s-\mu}}}\\
& \leq S^H_s - \la C_s \epsilon^{2s}|\log \epsilon| + o(\epsilon^{2s}) < S^H_s.
\end{split}
\end{equation}
If $n>4s$ then again using Proposition \ref{estimates}(iii), \eqref{esti-new} and \eqref{har-esti7}, we get
\begin{equation}\label{elem3}
\begin{split}
\frac{\|u_\epsilon\|^2 - \la \int_{\Om}|u_\epsilon|^2\mathrm{d}x}{\left(\int_{\Om}\int_{\Om}\frac{|u_\epsilon(x)|^{2^*_{\mu,s}}|u_\epsilon(y)|^{2^*_{\mu,s}}}{|x-y|^{\mu}}~\mathrm{d}x\mathrm{d}y \right)^{\frac{n-2s}{2n-\mu}}}
 & \leq \frac{(C(n,\mu))^{\frac{n(n-2s)}{2s(2n-\mu)}}(S^H_s)^{\frac{n}{2s}}- \la C_s \epsilon^{2s} + o(\epsilon^{n-2s})}{\left((C(n,\mu))^{\frac{n}{2s}}(S^H_s)^{\frac{2n-\mu}{2s}}- o(\epsilon^{n}) \right)^{\frac{n-2s}{2n-\mu}}}\\
& \leq S^H_s - \la C_s \epsilon^{2s}+ o(\epsilon^{n-2s}) < S^H_s.
\end{split}
\end{equation}
So \eqref{elem1} holds true if we take $w=u_\epsilon$. We have
\begin{equation*}
\begin{split}
&\max_{t\geq 0} \left( \frac{t^2\|w\|^2}{2} - \frac{t^{22^*_{\mu,s}}}{22^*_{\mu,s}}\int_{\Om}\int_{\Om}\frac{|w(x)|^{2^*_{\mu,s}}|w(y)|^{2^*_{\mu,s}}}{|x-y|^{\mu}}~\mathrm{d}x\mathrm{d}y
-\frac{\la t^2}{2}\int_{\Om}w^2\mathrm{d}x \right)\\
& = \frac{n+2s-\mu}{2(2n-\mu)}\left( \frac{\|w\|^2 - \la \int_{\Om}w^2\mathrm{d}x}{\left( \int_{\Om}\int_{\Om}\frac{|w(x)|^{2^*_{\mu,s}}|w(y)|^{2^*_{\mu,s}}}{|x-y|^{\mu}}~\mathrm{d}x\mathrm{d}y \right)^{\frac{n-2s}{2n-\mu}}}  \right)^{\frac{2n-\mu}{n+2s-\mu}} < \frac{n+2s-\mu}{2(2n-\mu)} (S^H_s)^{\frac{2n-\mu}{n+2s-\mu}}.
\end{split}
\end{equation*}
This implies
\[0 < \max_{t \geq 0}I(tw) < \frac{n+2s-\mu}{2(2n-\mu)} (S^H_s)^{\frac{2n-\mu}{n+2s-\mu}}.\]

From the definition $c^*$, we can say that $c^* < \frac{n+2s-\mu}{2(2n-\mu)} (S^H_s)^{\frac{2n-\mu}{n+2s-\mu}} $. Then, there exist a $(PS)$ sequence, say $\{u_k\}$ at $c^*$, using Proposition \ref{MP}. We know $\{u_k\}$ has a convergent subsequence, using Lemma \ref{PSlevel} and thus, $I$ has a critical value $c^* \in \left(0, \frac{n+2s-\mu}{2(2n-\mu)} (S^H_s)^{\frac{2n-\mu}{n+2s-\mu}} \right)$ which gives a nontrivial solution for $(P_\la)$.\hfill{\QED}

\subsection{Case (2):  $\la \geq \la_1$}
Let us consider the sequence of eigenvalues of the operator $(-\De)^s$ with homogenous Dirichlet boundary condition in $\mb R^n$, denoted by
\[0 < \la_1 < \la_2 \leq \la_3 \leq \ldots \leq \la_{j} \leq \la_{j+1}\leq \ldots\]
and $\{e_j\}_{j \in \mb N} \subset L^{\infty}(\Om)$ be the  corresponding sequence of eigenfunctions. We also consider this sequence of $e_j$'s to form an orthonormal basis of $X_0$. \\
In this case, without loss of generality, we can assume $\la \in [\la_r, \la_{r+1})$ for some $r \in \mb N$ and $e_r$ denote the eigenfunction corresponding to $\la_r$. We define
\[\mb M_{r+1} := \{u \in X_0: \langle u,e_i \rangle = \int_{Q}\frac{(u(x)-u(y))(e_i(x)-e_i(y))}{|x-y|^{n+2s}}\mathrm{d}x\mathrm{d}y =0,\; i= 1,2,\ldots,r\},\]
and
\[\mb D_{r} := \text{span} \{e_1,e_2,\ldots, e_r\}.\]
Clearly, $\mb D_r$ is finite dimensional and $\mb D_r \oplus \mb M_{r+1} = X_0$.

\begin{Lemma}\label{Link}
Let $n>2s$ and $\la \in [\la_r, \la_{r+1})$ for some $r \in \mb N$. Then the energy functional $I$ satisfies the following properties :
\begin{enumerate}
\item[(i)] There exists $\beta, \rho >0$ such that $I(u) \geq \beta$, for any $u \in \mb M_{r+1}$ with $\|u\|=\rho$.
\item[(ii)] If $u \in \mb D_{r}$, then $I(u)<0$.
\item[(iii)] If $\mb E$ is any finite dimensional subspace of $X_0$, then there exists $R> \rho$ such that for any $u \in \mb E$ with $\|u\| \geq R$, we have $I(u) \leq 0$.
\end{enumerate}
\end{Lemma}
\begin{proof}
\begin{enumerate}
\item[(i)] Since $\la \in [\la_r, \la_{r+1})$, using Sobolev embedding and Hardy-Littlewood-Sobolev inequality, we get
\begin{equation*}
\begin{split}
I(u)& \geq \frac12 \left(1-\frac{\la}{\la_{r+1}}\right)\|u\|^2 - \frac{1}{2 2^{*}_{\mu,s}}C_1|u|^{\frac{2(2n-\mu)}{n-2s}}_{2^*_s}\\
&\geq \frac12 \left(1-\frac{\la}{\la_{r+1}}\right)\|u\|^2 - \frac{1}{2 2^{*}_{\mu,s}}C_1C_2 \|u\|^{\frac{2(2n-\mu)}{n-2s}},
\end{split}
\end{equation*}
for all $u \in \mb M_{r+1} \setminus \{0\}$, where $C_1,C_2$ are positive constants. Since $0<\mu<n$, so $2< 2\left(\frac{2n-\mu}{n-2s}\right)$ and thus, some $\beta, \rho>0$ can be chosen such that $I(u)\geq \beta$ for $\|u\|= \rho$.
\item[(ii)] Let $u \in \mb D_{r}$, then there exists $a_i \in \mb R$ such that $u = \sum_{i=1}^{r} a_i e_i$. Since $e_j$'s form an orthonormal basis of $X_0$  and $L^2(\Om)$, we get
    \[\int_{\Om}u^2 ~\mathrm{d}x= \sum_{i=1}^{r}a_i^2\; \text{and} \; \|u\|^2 =\sum_{i=1}^{r}a_i^2\|e_i\|^2. \]
    This implies
    \begin{equation*}
    \begin{split}
    I(u)&= \frac12 \sum_{i=1}^{r}a_i^2(\|e_i\|^2-\la) -  \frac{1}{2 2^{*}_{\mu,s}} \int_{\Om}\int_{\Om}\frac{|u(x)|^{2^*_{\mu,s}}|u(y)|^{2^*_{\mu,s}}}{|x-y|^{\mu}}~\mathrm{d}x\mathrm{d}y\\
    & < \frac12 \sum_{i=1}^{r}a_i^2(\la_i-\la)\leq 0,
    \end{split}
    \end{equation*}
    because $\la \in [\la_r, \la_{r+1})$.
    \item[(iii)] We can assume $\mb E = \text{span} \{v_1,v_2,\ldots,v_k\}$. So, for every $v_i$, there exists a $t_i>0$ such that
    $I(tv_i) <0$, whenever $t > t_i$. Let $\hat{t}= \max\{t_1,t_2,\ldots,t_k\}$, then $I(tu)<0$ whenever $t>\hat{t}$ and $u \in \mb E$. Therefore, there exists $R> \rho$ such that for any $u \in \mb E$ with $\|u\| \geq R$, we have $I(u) \leq 0$.\hfill{\QED}
\end{enumerate}
\end{proof}
Now, we prove the fractional version of Lemma 4.2 of \cite{myang} following the same.
\begin{Lemma}\label{equiv-norm}
Let $n>2s$ and $\Om$ be a bounded domain in $\mb R^n$. Then
\[\|u\|_0 := \int_{\Om}\int_{\Om}\frac{|u(x)|^{2^*_{\mu,s}}|u(y)|^{2^*_{\mu,s}}}{|x-y|^{\mu}}~\mathrm{d}x\mathrm{d}y\]
for $u \in L^{2^*_s}(\Om)$, defines an equivalent norm on $L^{2^*_s}(\Om)$.
\end{Lemma}
\begin{proof}
Let $u,v \in L^{2^*_s}(\Om)$, then using H\"{o}lder inequality and semigroup property of Reisz potential, we get

\begin{equation*}
\begin{split}
&\int_{\Om}\int_{\Om}\frac{|(u+v)(x)|^{2^*_{\mu,s}}|(u+v)(y)|^{2^*_{\mu,s}}}{|x-y|^{\mu}}~\mathrm{d}x\mathrm{d}y = \int_{\Om}\int_{\Om}\frac{|(u+v)(x)|^{2 2^*_{\mu,s}}}{|x-y|^{\mu}}~\mathrm{d}x\mathrm{d}y\\
& \leq  \int_{\Om}\int_{\Om}\frac{|u(x)||(u+v)(x)|^{2 2^*_{\mu,s}-1}}{|x-y|^{\mu}}~\mathrm{d}x\mathrm{d}y + \int_{\Om}\int_{\Om}\frac{|v(x)||(u+v)(x)|^{2 2^*_{\mu,s}-1}}{|x-y|^{\mu}}~\mathrm{d}x\mathrm{d}y
\end{split}
\end{equation*}
\begin{equation*}
\begin{split}
& \leq \left(\int_{\Om}\int_{\Om}\frac{|u(x)|^{2 2^*_{\mu,s}}}{|x-y|^{\mu}}~\mathrm{d}x\mathrm{d}y\right)^{\frac{1}{22^{*}_{\mu,s}}}\left(\int_{\Om}\int_{\Om}\frac{|(u+v)(x)|^{2 2^*_{\mu,s}}}{|x-y|^{\mu}}~\mathrm{d}x\mathrm{d}y\right)^{1-\frac{1}{22^{*}_{\mu,s}}}\\
& \quad +\left(\int_{\Om}\int_{\Om}\frac{|v(x)|^{2 2^*_{\mu,s}}}{|x-y|^{\mu}}~\mathrm{d}x\mathrm{d}y\right)^{\frac{1}{22^{*}_{\mu,s}}}\left(\int_{\Om}\int_{\Om}\frac{|(u+v)(x)|^{2 2^*_{\mu,s}}}{|x-y|^{\mu}}~\mathrm{d}x\mathrm{d}y\right)^{1-\frac{1}{22^{*}_{\mu,s}}}\\
& = \left(\int_{\Om}\int_{\Om}\frac{|(u+v)(x)|^{2 2^*_{\mu,s}}}{|x-y|^{\mu}}~\mathrm{d}x\mathrm{d}y\right)^{1-\frac{1}{22^{*}_{\mu,s}}}\\
&\quad \quad \times \left(\left(\int_{\Om}\int_{\Om}\frac{|u(x)|^{2 2^*_{\mu,s}}}{|x-y|^{\mu}}~\mathrm{d}x\mathrm{d}y\right)^{\frac{1}{22^{*}_{\mu,s}}} + \left(\int_{\Om}\int_{\Om}\frac{|v(x)|^{2 2^*_{\mu,s}}}{|x-y|^{\mu}}~\mathrm{d}x\mathrm{d}y\right)^{\frac{1}{22^{*}_{\mu,s}}} \right).
\end{split}
\end{equation*}
Therefore, we get $\|u+v\|_0 \leq \|u\|_0 + \|v\|_0$ and other properties of norm are also satisfied by $\|\cdot\|_0$. So, $\|\cdot\|_0$ is a norm on $L^{2^*_s}(\Om)$ and $L^{2^*_s}(\Om)$ is a Banach space under this norm(proof can be sketched using the techniques to prove $L^p(\Om)$ is a Banach space with the usual $L^p$-norm). By Hardy-Littlewood-Sobolev inequality, we have
\[ \|u\|_0 \leq (C(n,\mu))^{\frac{1}{22^*_{\mu,s}}}|u|_{2^*}. \]
So, the identity map from $(L^{2^*_{\mu,s}}(\Om), \|\cdot\|_0)$ to $(L^{2^*_{\mu,s}}(\Om), |\cdot|_{2^*_{\mu,s}})$ is linear and bounded. Thus, by open mapping theorem, we obtain $\|\cdot\|_0$ is equivalent norm to the standard norm $|\cdot|_{2^*_{\mu,s}}$ on $L^{2^*_{\mu,s}}(\Om)$.\hfill{\QED}
\end{proof}

\noi Before proceeding further, we define the linear space
\[\mb G_{r,\epsilon} := \text{span} \{e_1,e_2, \ldots, e_r, u_\epsilon\}\]
and set
\[g_{r, \epsilon} :=  \max_{u \in M} \left( \|u\|^2- \la \int_{\Om}|u|^2~\mathrm{d}x \right),\]
where $ M = \{u \in \mb G_{r,\epsilon} : \int_{\Om}\int_{\Om}\frac{|u(x)|^{2^*_{\mu,s}}|u(y)|^{2^*_{\mu,s}}}{|x-y|^{\mu}}~\mathrm{d}x\mathrm{d}y =1\}$ and $u_\epsilon$ (from \eqref{elem1}) is such that
\begin{equation*}
\frac{\|u_\epsilon\|^2 - \la \int_{\Om}|u_\epsilon|^2\mathrm{d}x}{\left(\int_{\Om}\int_{\Om}\frac{|u_\epsilon(x)|^{2^*_{\mu,s}}|u_\epsilon(y)|^{2^*_{\mu,s}}}{|x-y|^{\mu}}~\mathrm{d}x\mathrm{d}y \right)^{\frac{n-2s}{2n-\mu}}} < S^H_s.
\end{equation*}
\begin{Lemma}\label{g-esti}
Let $n\geq 2s$ and $\la \in [\la_r,\la_{r+1})$ for some $r \in \mb N$, then the following holds true:
\begin{enumerate}
\item[(i)] There exist $u_g \in \mb G_{r,\epsilon}$ such that $g_{r,\epsilon}$ is achieved at $u_g$ and
\[u_g = w + tu_\epsilon\]
with $w \in \mb D_r$ and $t \geq 0$.
\item[(ii)] As $\epsilon \rightarrow 0$, we have
$$g_{r,\epsilon}=
\left\{
	\begin{array}{ll}
		(\la_r-\la)|w|^2_2 & \mbox{if } t=0 \\
		(\la_r-\la)|w|^2_2+ F_\epsilon(1+|w|_2 o(\epsilon^{\frac{n-2s}{2}}))+ o(\epsilon^{\frac{n-2s}{2}})|w|_2 & \mbox{if } t>0,
	\end{array}
\right.$$
where $w $ is defined in (i) and $F_\epsilon$ is given by
\[F_\epsilon = \frac{\|u_\epsilon\|^2 - \la \int_{\Om}|u_\epsilon|^2\mathrm{d}x}{\left(\int_{\Om}\int_{\Om}\frac{|u_\epsilon(x)|^{2^*_{\mu,s}}|u_\epsilon(y)|^{2^*_{\mu,s}}}{|x-y|^{\mu}}~\mathrm{d}x\mathrm{d}y \right)^{\frac{n-2s}{2n-\mu}}} .\]
\end{enumerate}
\end{Lemma}
\begin{proof}
\begin{enumerate}
\item[(i)] Clearly $\mb G_{r,\epsilon}$ is finite dimensional, so $g_{r,\epsilon}$ is achieved at $u_g$, say. Then, $u_g \in M$ and by definition of $\mb G_{r,\epsilon}$, there exist $w \in \mb D_r$ and $t\in \mb R$ such that
    $u_g = w + tu_\epsilon$.
We can assume $t \geq 0$ because if $t<0$, then we can replace $u_g$ by $-u_g$.
\item[(ii)] To prove this, first let $t=0$, then $u_g= w \in \mb D_r$ and
\[g_{r,\epsilon} = \|w\|^2 - \la \int_{\Om}|w|^2~\mathrm{d}x \leq (\la_r-\la)|w|_2^2. \]
Now, suppose $t>0$ and set
\[\widehat{w} = w+ t \sum_{i=1}^{r}\left( \int_{\Om}u_\epsilon e_i\mathrm{d}x\right)e_i \in \mb D_r\;\; \text{and} \; \; \widehat{u}_\epsilon= u_\epsilon - \sum_{i=1}^{r}\left( \int_{\Om}u_\epsilon e_i\mathrm{d}x\right)e_i \]
and find that $\widehat{w}$ and $ \widehat{u}_{\epsilon}$ are orthogonal in $L^2(\Om)$.
Then, $u_g = \widehat{w}+t \widehat{u}_{\epsilon}$ and $|u_g|^2_2 = |\widehat{w}|^2_2+t^2|\widehat{u}_{\epsilon}|^2_2$. Since
\[\int_{\Om}\int_{\Om}\frac{|u_g(x)|^{2^*_{\mu,s}}|u_g(y)|^{2^*_{\mu,s}}}{|x-y|^{\mu}}~\mathrm{d}x\mathrm{d}y =1,\]
 using lemma \ref{equiv-norm}, we get a constant $C_0>0$(independent of $\epsilon$) such that $|u_g|_{2^{*}_{\mu,s}} \leq C_0$. Subsequently, using H\"{o}lder inequality, we get a constant $C_1>0$(also independent of $\epsilon$) such that $|u_g|^2_2 \leq C_1$. Therefore, we can find $C_2>0$ such that $|u_g|^2_2$ and $|\widehat{w}|^2_2$ are both uniformly bounded in $\epsilon$. This further implies that $t < C_3$, for some $C_3>0$. By computations as before, we get
 \begin{equation}\label{g-esti1}
 \begin{split}
 |u_\epsilon|^{\frac{3n-2\mu+2s}{n-2s}}_{\frac{n(3n-2\mu+2s)}{(2n-\mu)(n-2s)}} &= \left( \int_{\Om} |u_\epsilon|^{\frac{n(3n-2\mu+2s)}{(2n-\mu)(n-2s)}}~\mathrm{d}x\right)^{\frac{2n-\mu}{n}} \leq \left( \int_{B_{2\delta}} |U_\epsilon|^{\frac{n(3n-2\mu+2s)}{(2n-\mu)(n-2s)}}~\mathrm{d}x\right)^{\frac{2n-\mu}{n}}\\
 & \leq C_4 \epsilon^{\frac{n-2s}{2}} \left( \int_{0}^{\frac{2\delta}{\epsilon}}\frac{r^{n-1}}{(1+r^2)^{\frac{n(3n-2\mu+2s)}{(2n-\mu)(n-2s)}}}~\mathrm{d}r \right)^{\frac{2n-\mu}{n}} \leq o(\epsilon^{\frac{n-2s}{2}}),
 \end{split}
 \end{equation}
  where $C_4>0$ is constant. Since $e_1, e_2, \ldots, e_r \in L^{\infty}(\Om)$, we have $\hat{w} \in L^{\infty}(\Om)$. Using the fact that the map $t \mapsto t^{22^*_{\mu,s}}$ in convex, for $t \geq 0$ and $\mb D_{r}$ being finite dimensional, all norms are equivalent, we get
  \begin{equation*}
 \begin{split}
 1&=\int_{\Om}\int_{\Om}\frac{|u_g(x)|^{2^*_{\mu,s}}|u_g(y)|^{2^*_{\mu,s}}}{|x-y|^{\mu}}~\mathrm{d}x\mathrm{d}y
 = \int_{\Om}\int_{\Om}\frac{|(w+tu_\epsilon)(x)|^{2 2^*_{\mu,s}}}{|x-y|^{\mu}}~\mathrm{d}x\mathrm{d}y\\
 &\geq  \int_{\Om}\int_{\Om}\frac{|tu_\epsilon(x)|^{2 2^*_{\mu,s}}}{|x-y|^{\mu}}~\mathrm{d}x\mathrm{d}y + 2 2^*_{\mu,s} \int_{\Om}\int_{\Om}\frac{|tu_\epsilon(x)|^{2 2^*_{\mu,s}-1}|w(x)|}{|x-y|^{\mu}}~\mathrm{d}x\mathrm{d}y
 \end{split}
 \end{equation*}
 \begin{equation*}
 \begin{split}
 & \geq \int_{\Om}\int_{\Om}\frac{|tu_\epsilon(x)|^{ 2^*_{\mu,s}}|tu_\epsilon(y)|^{ 2^*_{\mu,s}}}{|x-y|^{\mu}}~\mathrm{d}x\mathrm{d}y\\
  & \quad - 2 2^*_{\mu,s}|w|_{\infty} \int_{\Om}\int_{\Om}\frac{|tu_\epsilon(x)|^{\frac{2 2^*_{\mu,s}-1}{2}} |tu_\epsilon(y)|^{\frac{2 2^*_{\mu,s}-1}{2}}}{|x-y|^{\mu}}~\mathrm{d}x\mathrm{d}y\\
 & \geq \int_{\Om}\int_{\Om}\frac{|tu_\epsilon(x)|^{ 2^*_{\mu,s}}|tu_\epsilon(y)|^{ 2^*_{\mu,s}}}{|x-y|^{\mu}}~\mathrm{d}x\mathrm{d}y - C_5 |w|_2 |u_\epsilon|^{\frac{3n-2\mu+2s}{n-2s}}_{\frac{n(3n-2\mu+2s)}{(2n-\mu)(n-2s)}}.
 \end{split}
 \end{equation*}
 Considering \eqref{g-esti1} with above inequality, we get
 \[ \int_{\Om}\int_{\Om}\frac{|tu_\epsilon(x)|^{ 2^*_{\mu,s}}|tu_\epsilon(y)|^{ 2^*_{\mu,s}}}{|x-y|^{\mu}}~\mathrm{d}x\mathrm{d}y \leq 1+ C_5 |w|_2o(\epsilon^{\frac{n-2s}{2}}).\]
 Hence, using the definition of $A_\epsilon$ and $v$ being linear combination of finitely many eigenfunctions, we get
 \begin{equation*}
 \begin{split}
 g_{r,\epsilon} &\leq (\la_r -\la)|w|_2^2 + A_\epsilon \left(\int_{\Om}\int_{\Om}\frac{|u_\epsilon(x)|^{2^*_{\mu,s}}|u_\epsilon(y)|^{2^*_{\mu,s}}}{|x-y|^{\mu}}~\mathrm{d}x\mathrm{d}y \right)^{\frac{n-2s}{2n-\mu}} + C_6 \int_{\Om}u_\epsilon w\mathrm{d}x\\
 & \leq (\la_r -\la)|w|_2^2 + A_\epsilon \left(  1+ C_5 |w|_2o(\epsilon^{\frac{n-2s}{2}}) \right)+ C_7 |u_\epsilon|_1 |w|_2\\
 & \leq (\la_r -\la)|w|_2^2 + A_\epsilon \left(  1+ C_5 |w|_2o(\epsilon^{\frac{n-2s}{2}}) \right) + o(\epsilon^{\frac{n-2s}{2}}) |w|_2,
 \end{split}
 \end{equation*}
\end{enumerate}
where we used $|u_\epsilon|_1= o(\epsilon^{\frac{n-2s}{2}})$  (which can be derived as other estimates done before). This completes the proof.\hfill{\QED}
\end{proof}

\begin{Lemma}\label{G-min}
If $n \geq 4s$ and $\la \in [\la_r, \la_{r+1})$, for some $r \in \mb N$, then for every $u \in G_{r, \epsilon}$ we have
\[\frac{\|u\|^2 - \la \int_{\Om}|u|^2\mathrm{d}x}{\left(\int_{\Om}\int_{\Om}\frac{|u(x)|^{2^*_{\mu,s}}|u(y)|^{2^*_{\mu,s}}}{|x-y|^{\mu}}~\mathrm{d}x\mathrm{d}y \right)^{\frac{n-2s}{2n-\mu}}} < S^H_s.\]
\end{Lemma}
\begin{proof}
It is enough to show that $g_{r,\epsilon} < S^H_s$. From lemma \ref{g-esti}, if $t=0$ we have
\[ g_{r,\epsilon} \leq (\la_r -\la)|w|^2_2 <0 < S^H_s.\]
Else if $t>0$, then we consider the cases $n=4s$ and $n>4s$ separately.\\
\textbf{Case: $(n=4s)$} By lemma \ref{g-esti}(ii) and estimates in \eqref{elem2}, we have
\begin{equation*}
\begin{split}
g_{r,\epsilon} &\leq (\la_r -\la)|w|^2_2+ \frac{\|u_\epsilon\|^2 - \la \int_{\Om}|u_\epsilon|^2\mathrm{d}x}{\left(\int_{\Om}\int_{\Om}\frac{|u_\epsilon(x)|^{2^*_{\mu,s}}|u_\epsilon(y)|^{2^*_{\mu,s}}}{|x-y|^{\mu}}~\mathrm{d}x\mathrm{d}y \right)^{\frac{2s}{8s-\mu}}}(1+ |w|_2)o(\epsilon^s) +o(\epsilon^s)|w|_2\\
& \leq S^H_s - \la C_s\epsilon^{2s}|\log \epsilon|+ o(\epsilon^{2s}) +(\la_r -\la)|w|^2_2+ |w|_2o(\epsilon^s),
\end{split}
\end{equation*}
for sufficiently small $\epsilon>0$.\\
\textbf{Case: $(n>4s)$} Again, by lemma \ref{g-esti}(ii) and estimates in \eqref{elem3}, we have
\begin{equation*}
\begin{split}
g_{r,\epsilon} &\leq (\la_r -\la)|w|^2_2+ \frac{\|u_\epsilon\|^2 - \la \int_{\Om}|u_\epsilon|^2\mathrm{d}x}{\left(\int_{\Om}\int_{\Om}\frac{|u_\epsilon(x)|^{2^*_{\mu,s}}|u_\epsilon(y)|^{2^*_{\mu,s}}}{|x-y|^{\mu}}~\mathrm{d}x\mathrm{d}y \right)^{\frac{n-2s}{2n-\mu}}}(1+ |w|_2)o(\epsilon^{\frac{n-2s}{2}}) +o(\epsilon^{\frac{n-2s}{2}})|w|_2\\
& \leq S^H_s - \la C_s\epsilon^{2s}+ o(\epsilon^{n-2s}) +(\la_r -\la)|w|^2_2+ |w|_2o(\epsilon^{\frac{n-2s}{2}})
\end{split}
\end{equation*}
for sufficiently small $\epsilon>0$. Also, we have that
\[(\la_r -\la)|w|^2_2+ |w|_2o(\epsilon^{\frac{n-2s}{2}}) \leq \frac{1}{4(\la_r-\la)}o(\epsilon^{{n-2s}})= o(\epsilon^{n-2s},)\]
which implies $g_{r,\epsilon} < S^H_s$ for both the cases. This completes the proof. \hfill{\QED}
\end{proof}

\noi \textbf{Proof of Theorem \ref{thrm1}: ($n\geq 4s$, $\la >\la_1$)}\\
  In the proof of lemma \ref{g-esti}(ii), we considered
\[\widehat{u}_\epsilon= u_\epsilon - \sum_{i=1}^{r}\left( \int_{\Om}u_\epsilon e_i\mathrm{d}x\right)e_i.\]
From the definition of $\mb G_(r,\epsilon)$, we can write that
\[\mb G_{r,\epsilon} = \mb D_{r} \oplus u_\epsilon \mb R = \mb D_{r} \oplus \hat{u}_\epsilon \mb R, \]
where $u_\epsilon\mb R =\{ru_\epsilon: r \in \mb R\}$ and similarly, $z_\epsilon\mb R $. By lemma \ref{Link}, we have
\begin{enumerate}
\item[(i)] $\inf_{u \in \mb M_{r+1}, \|u\|=\rho} I(u) \geq \beta >0$,
\item[(ii)] $ \sup_{u \in \mb D_{r}} I(u) <0$, and
\item[(iii)] $\sup_{u \in \mb G_{r,\epsilon}, \|u\|\geq R}I(u) \leq 0 $,
\end{enumerate}
where $\beta, \rho$ are defined in lemma \ref{Link}. Therefore, $I$ satisfies the geometric structure of the linking theorem (Theorem 5.3,\cite{rabi}). We define
\[\bar{c} = \inf_{\gamma \in \Gamma}\max_{u \in A}I(\gamma(u)) >0,\]
where $\gamma := \{\gamma \in C(\bar{A},X_0): \gamma = id \text{ on } \partial A\}$ and $ A := (\bar{B}_R \cap \mb D_r ) \oplus \{r\hat{u}_\epsilon: r \in (0,R)\}$. By definition, for any $\gamma \in \Gamma$, we have
$\bar{c} \leq \max_{u \in A}I(\gamma(u))$
and particularly, if we take $\gamma = id$ on $\bar{A}$, then
\[\bar{c} \leq \max_{u \in A}I(u) \leq \max_{\mb G_{r,\epsilon}} I(u).\]
As we earlier saw, for any $u \in X_0 \setminus \{0\}$,
\begin{equation}\label{G-min1}
\frac{n+2s-\mu}{2(2n-\mu)}\left( \frac{\|w\|^2 - \la \int_{\Om}|w|^2\mathrm{d}x}{\left( \int_{\Om}\int_{\Om}\frac{|w(x)|^{2^*_{\mu,s}}|w(y)|^{2^*_{\mu,s}}}{|x-y|^{\mu}}~\mathrm{d}x\mathrm{d}y \right)^{\frac{n-2s}{2n-\mu}}}  \right)^{\frac{2n-\mu}{n+2s-\mu}} = \max_{t \geq 0}I(tu).
\end{equation}
Since $\mb G_{r,\epsilon}$ is a linear space, we have
\[\max_{u \in \mb G_{r,\epsilon}} I(u)= \max_{u \in \mb G_{r,\epsilon}, t\neq 0} I(\frac{|t|u}{|t|}) \leq \max_{u \in \mb G_{r,\epsilon}, t\geq 0} I(tu). \]
Hence, using lemma \ref{G-min} and \eqref{G-min1}, we get
\begin{equation*}
\bar{c} \leq \max_{u \in \mb G_{r,\epsilon}, t\geq 0} I(tu)< \frac{n+2s-\mu}{2(2n-\mu)} (S^H_s)^{\frac{2n-\mu}{n+2s-\mu}}.
\end{equation*}
Finally, using Linking theorem and lemma \ref{PSlevel}, we conclude that $(P_\la)$ has a nontrivial solution in $X_0$ with critical value $\bar{c} \geq \beta$.\hfill{\QED}
\section{Proof of Theorem \ref{newthrm}}
We will prove this theorem using the Mountain Pass and Linking Theorem in a combined way.
\begin{Lemma}
Let $2s<n<4s$ and $u_\epsilon$ be as defined in section 4, case 1. Then there exists $\bar \la>0$ such that for $\la > \bar \la$,
\[\frac{\|u_\epsilon\|^2 - \la \int_{\Om}|u_\epsilon|^2\mathrm{d}x}{\left(\int_{\Om}\int_{\Om}\frac{|u_\epsilon(x)|^{2^*_{\mu,s}}|u_\epsilon(y)|^{2^*_{\mu,s}}}{|x-y|^{\mu}}~\mathrm{d}x\mathrm{d}y \right)^{\frac{n-2s}{2n-\mu}}} < S^H_s.\]
\end{Lemma}
\begin{proof}
Using Proposition \ref{estimates} and \ref{estimates1}, we get
\begin{equation*}
\begin{split}
\frac{\|u_\epsilon\|^2 - \la \int_{\Om}|u_\epsilon|^2\mathrm{d}x}{\left(\int_{\Om}\int_{\Om}\frac{|u_\epsilon(x)|^{2^*_{\mu,s}}|u_\epsilon(y)|^{2^*_{\mu,s}}}{|x-y|^{\mu}}~\mathrm{d}x\mathrm{d}y \right)^{\frac{n-2s}{2n-\mu}}}
 & \leq \frac{(C(n,\mu))^{\frac{n(n-2s)}{2s(2n-\mu)}}(S^H_s)^{\frac{n}{2s}}- \la C_s \epsilon^{n-2s} + o(\epsilon^{2s})}{\left((C(n,\mu))^{\frac{n}{2s}}(S^H_s)^{\frac{2n-\mu}{2s}}- o(\epsilon^{n}) \right)^{\frac{n-2s}{2n-\mu}}}\\
& \leq S^H_s + \frac{\epsilon^{n-2s}(o(1)-\la C_s)}{((C(n,\mu))^{\frac{n}{2s}}(S^H_s)^{\frac{2n-\mu}{2s}}-o(\epsilon^n))^{\frac{n-2s}{2n-\mu}}}+o(\epsilon^{2s})\\
&< S^H_s,
\end{split}
\end{equation*}
when we choose $\la>0$ large enough, say $\la > \bar \la$ and provided $\epsilon>0$ be sufficiently small. This completes the proof.\QED
\end{proof}

We have already seen in previous sections that the functional $I$ satisfies geometry of Mountain Pass  when $\la < \la_1$ (using Lemma \ref{MGP}). When $\la \geq \la_1$, without loss of generality, we assume $\la \in [\la_r, \la_{r+1})$, for some $r \in \mb N$. Then using Lemma \ref{Link}, we get that  $I$ satisfies geometry of Linking theorem. Also, by Lemma \ref{PSlevel}, we get that $I$ satisfies the $(PS)_c$ condition when
\[c < \frac{n+2s-\mu}{2(2n-\mu)})=(S^H_s)^{\frac{2n-\mu}{n+2s-\mu}}.\]
So, in order to apply the classical critical point theorems, we need the Mountain Pass critical level and Linking critical level of $I$ to stay below this threshold. Consider $\mb M_{r+1}, \mb D_r$ and $G_{r,\epsilon}$ be as defined in earlier section. Note that Lemma \ref{g-esti} holds true in this case and we have the following lemma.

\begin{Lemma}
If $2s<n<4s$ and $\la \in [\la_r, \la_{r+1})$, for some $r \in \mb N$, then for every $u \in G_{r, \epsilon}$ we have
\[\frac{\|u\|^2 - \la \int_{\Om}|u|^2\mathrm{d}x}{\left(\int_{\Om}\int_{\Om}\frac{|u(x)|^{2^*_{\mu,s}}|u(y)|^{2^*_{\mu,s}}}{|x-y|^{\mu}}~\mathrm{d}x\mathrm{d}y \right)^{\frac{n-2s}{2n-\mu}}} < S^H_s.\]
\end{Lemma}
\begin{proof}
If $t=0$ then since $\la \in [\la_r, \la_{r+1})$, we get
\[g_{r,\epsilon} \leq (\la_r-\la)|w|^2_2 \leq 0 < S^H_s.\]
When $t>0$, then
\begin{equation*}
\begin{split}
g_{r,\epsilon} &\leq (\la_r -\la)|w|^2_2+ \frac{\|u_\epsilon\|^2 - \la \int_{\Om}|u_\epsilon|^2\mathrm{d}x}{\left(\int_{\Om}\int_{\Om}\frac{|u_\epsilon(x)|^{2^*_{\mu,s}}|u_\epsilon(y)|^{2^*_{\mu,s}}}{|x-y|^{\mu}}~\mathrm{d}x\mathrm{d}y \right)^{\frac{2s}{8s-\mu}}}(1+ |w|_2)o(\epsilon^{\frac{n-2s}{2}}) +o(\epsilon^{\frac{n-2s}{2}})|w|_2\\
& \leq S^H_s +\left( \frac{(o(1)-\la C_s)\epsilon^{n-2s}}{\left((C(n,\mu))^{\frac{n}{2s}}(S^H_s)^{\frac{2n-\mu}{2s}}- o(\epsilon^{n})\right)^{\frac{n-2s}{2n-\mu}}}+ o(\epsilon^{2s})\right)(1+ |w|_2)o(\epsilon^{\frac{n-2s}{2}})\\
 & \quad \quad +(\la_r -\la)|w|^2_2+ |w|_2o(\epsilon^{\frac{n-2s}{2}})\\
& \leq S^H_s +\frac{(o(1)-\la C_s)\epsilon^{n-2s}}{\left((C(n,\mu))^{\frac{n}{2s}}(S^H_s)^{\frac{2n-\mu}{2s}}- o(\epsilon^{n})\right)^{\frac{n-2s}{2n-\mu}}}+ (\la_r -\la)|w|^2_2 + |w|_2o(\epsilon^{\frac{n-2s}{2}}) < S^H_s,
\end{split}
\end{equation*}
for sufficiently small $\epsilon >0$ because we consider $\la > \bar \la$ and $\la \in (\la_r, \la_{r+1})$. Hence the result follows.\QED
\end{proof}

\noi \textbf{Proof of Theorem \ref{newthrm}:} We consider two cases:\\
\textbf{Case 1.} $(\la_1 > \bar \la)$ For this case, we use Mountain Pass theorem if $\la \in (\bar \la, \la_1)$ and Linking theorem if $\la \in (\la_r,\la_{r+1})$ for some $r \in \mb N$.

If $\la \in (\bar \la , \la_1)$, using Remark \ref{rmrk}, Lemma \ref{MGP} and Proposition \ref{MP}, following the same arguments as Case 1 in proof of Theorem \ref{thrm1}, we get that $(P_\la)$ admits a nontrivial solution.

Otherwise if $(\la_1 > \bar \la)$ , we assume $\la \in (\la_r,\la_{r+1})$ for some $r \in \mb N$ (since $\la$ is not an eigenvalue of $(-\De)^s$). Here, following the arguments as in Case 2 in proof of Theorem \ref{thrm1}, we get that $(P_\la)$ admits a nontrivial solution.\\
\textbf{Case 2.} $(\la_1 < \bar \la)$ In this case, we can assume $\la \in (\la_r,\la_{r+1})$ for some $r \in \mb N$ and $\la > \bar \la$. Here again, following the arguments as in Case 2 in proof of Theorem \ref{thrm1}, we get that $(P_\la)$ admits a nontrivial solution.

\section{Multiplicity Results}
By the equivalence of norms obtained in lemma \ref{equiv-norm}, we can find a constant $C^{\prime}>0$ such that
\[C^{\prime} |u|_{2^*_s} \leq \left( \int_{\Om}\int_{\Om}\frac{|u(x)|^{2^*_{\mu,s}}|u(y)|^{2^*_{\mu,s}}}{|x-y|^{\mu}}~\mathrm{d}x\mathrm{d}y\ \right)^{\frac{1}{2 2^*_{\mu,s}}},\]
for all $u \in L^{2^*_s}(\Om)$. Let us define
\[\la_* := \frac{S^H_s (C^{\prime})^2}{|\Om|^{\frac{2s}{n}}}\]
and we consider the set containing the eigenvalues between $\la$ and $\la + \bar{\la}$, that is
\[\Upsilon = \{\la < \la_i < \la_*\} = \{\la_{k+1}, \la_{k+2}, \ldots, \la_{k+q}\}.\]
If $\Upsilon$ is not empty, then we can prove Theorem \ref{thrm2}.\\
Let $V$ be a Banach space, we define
\[\sum := \{E \subset V\setminus \{0\}: E \text{ is closed in } V \text{ and symmetric with respect to origin}\}.\]
We also define genus of the set $E \in \sum$ as
\[\gamma(E) := \inf\{k \in \mb N: \exists \; \varphi \in C(E,\mb R^k)\setminus\{0\}, \varphi(x)=-\varphi(y)\}.\]
 Also, $\gamma(E) =+ \infty$, if there exists no $\varphi$ as given in definition above. We give the definition of pseudo-index.
\begin{Definition}\cite{myang}
For $E \in \sum^*= \{A \in \sum; A \text{ is compact}\}$ and
\[ \Lambda_*(\rho) = \{h \in C(V,V); h \text{ is an odd homeomorphism and } h(B_1) \subset I^{-1}(0, +\infty) \cup B_\rho\},\]
we define $i^*(E) = \inf_{h \in \Lambda_*(\rho)} \gamma (E \cap h(\partial B_1))$, for any $\rho>0$.
\end{Definition}
We state some necessary results (without giving their proofs) from \cite{myang} that will  help us to conclude our main theorem.
\begin{Proposition}\label{br1}
\begin{enumerate}
\item[(i)] Let $t \in \mb N$ and $Y$   be a subspace of $V$ with codimension $t$ and $E \subset \sum$ with $\gamma(E)>t$, then $E \cap Y \neq \emptyset$.
    \item[(ii)] If $E \subset V$, $\Om$ is a bounded neighborhood of $0$ in $\mb R^t$, and there exists a mapping $h \in C(E, \partial \Om)$ with $h$ an odd homeomorphism, then $\gamma(E)=t$.
        \item[(iii)]  If $\gamma (E) = t$ and $0 \not\in E$, then $E$ contains at least $t$ distinct pairs of points.
\end{enumerate}
\end{Proposition}
\begin{Lemma}\label{br2}
Let $V$ be a Banach space and $I \in C(V, \mb R)$ be an even functional satisfying:
\begin{enumerate}
\item[(i)] There exist $\rho, \beta>0$ and $V_1 \subset V$ with $dim V_1 = t$ such that $I\mid_{\partial B_\rho\cap V_1^\perp }  \geq \beta$.
    \item[(ii)]There exist $V_2 \subset V$ with $dimV_2 = t_1>t$ and $R>0$ such that for any $u \in V_2\setminus B_R$, $I(u) \leq 0$.
\end{enumerate}
We define $c^*_k := \inf\{\sup_{u \in A} I(u): A \in \sum^*, i^*(A)\geq k\}$. If $0 < c^*_{k+1}\leq c^*_{k+2}\leq \ldots \leq c^*_{m} < +\infty$ and $I$ satisfies the $(PS)_{c^*_i}$ condition  at $c^*_i$ $(k+1 \leq i \leq m )$, then $I$ has atleast $m-k$ distinct pairs of critical points and $c^*_i(k+1 \leq i \leq m)$ is the corresponding critical value.
\end{Lemma}
\begin{Lemma}\label{br3}
If $n >2s$ and $\la < \la_{j+1}$ for some $j \in \mb N$, then the energy functional $I$ satisfies the following:
 \begin{enumerate}
\item[(i)] There exists $\beta, \rho >0$ such that $I(u) \geq \beta$, for any $u \in \mb D_j^{\perp}$ with $\|u\|=\rho$.
\item[(iii)] If $\mb E$ be any finite dimensional subspace of $X_0$, then there exists $R> \rho$ such that for any $u \in \mb E$ with $\|u\| \geq R$, we have $I(u) \leq 0$.
\end{enumerate}
\end{Lemma}
\begin{proof}
Proof follows similar to proof of lemma \ref{Link}.\hfill{\QED}
\end{proof}

\begin{Lemma}\label{br4}
The following holds, for $1\leq m\leq q$,
\[\beta \leq c^*_{j+m}<  \frac{n+2s-\mu}{2(2n-\mu)} (S^H_s)^{\frac{2n-\mu}{n+2s-\mu}}.\]
\end{Lemma}
\begin{proof}
Let $A \in \sum^*$ and $i^*(A) \geq j+m$. We set $f=\rho.id$, where $\rho$ is obtained in Lemma \ref{br2} and $id$ is the identity map. Then it can be easily checked that $f \in \Gamma_* $ and
\[\gamma(A \cap \partial B_\rho) = \gamma(A \cap \partial f(\partial B_1)) \geq \inf_{f \in \Gamma_*(\rho)}\gamma(A \cap f(\partial B_1)) = i^*(A) \geq j+m. \]
Thus, using Proposition \ref{br1}(i), we get $A \cap \partial B_\rho \cap \mb D_r^{\perp} \neq \emptyset$. Then lemma \ref{br3}(i) gives
\[\sup_{u \in A}I(u) \geq \inf_{u \in \partial B_\rho \cap \mb D_r^{\perp}}I(u) \geq \beta.\]
Since $A$ is arbitrary, $\beta \leq c^*_{j+m}$. Now, we define $\tilde{A}= \mb D_{j+m} \cap \bar{B_R} \in \sum^*$. So, for any $f \in \Gamma_*(\rho)(0< \rho < R)$, we have
\[\tilde{A} \supset  \mb D_{j+m} \cap (I^{-1}(0,+\infty) \cup B_\rho) \supset \mb D_{j+m} \cap h(B_1). \]
Using definition of pseudo-index, $i^*(\tilde A) \geq j+m$ and from definition of $c^*_{j+m}$, we get $c^*_{j+m} \leq \sup_{u \in \tilde{A}} I(u)$. Using compactness of $\tilde{A}$, we obtain $\tilde{u} \in \tilde{A}$ such that
\[I(\tilde{u})= \sup_{u \in \tilde{A}} I(u).\]
Hence, $c^*_{j+m} \leq I(\tilde{u}) = \max_{t>0}I(t\tilde{u})$. Now, using the value of $\la_*$, Sobolev embedding, Hardy-Littlewood-Sobolev inequality and the fact that $\tilde{u} \in \tilde{A}$, we have
\begin{equation*}
\begin{split}
\max_{t\geq 0}I(tu) &= \frac{n+2s-\mu}{2(2n-\mu)}\left( \frac{\|\tilde{u}\|^2 - \la \int_{\Om}|\tilde{u}|^2\mathrm{d}x}{\left( \int_{\Om}\int_{\Om}\frac{|\tilde{u}(x)|^{2^*_{\mu,s}}|\tilde{u}(y)|^{2^*_{\mu,s}}}{|x-y|^{\mu}}~\mathrm{d}x\mathrm{d}y \right)^{\frac{n-2s}{2n-\mu}}}  \right)^{\frac{2n-\mu}{n+2s-\mu}}\\
& \leq \frac{n+2s-\mu}{2(2n-\mu)}\left( \frac{(\la_{k+m-\la})\int_{\Om}\tilde{u}^2\mathrm{d}x}{\left( \int_{\Om}\int_{\Om}\frac{|\tilde{u}(x)|^{2^*_{\mu,s}}|\tilde{u}(y)|^{2^*_{\mu,s}}}{|x-y|^{\mu}}~\mathrm{d}x\mathrm{d}y \right)^{\frac{n-2s}{2n-\mu}}}  \right)^{\frac{2n-\mu}{n+2s-\mu}}\\
& \leq \frac{n+2s-\mu}{2(2n-\mu)}\left( \frac{\la_*|\Om|^{\frac{2s}{n}}\left(\int_{\Om}\tilde{u}^{2^*_s}\right)^{\frac{2}{2^*_s}}}{(C^{\prime})^2\left(\int_{\Om}\tilde{u}^{2^*_s}\right)^{\frac{2}{2^*_s}}}  \right)^{\frac{2n-\mu}{n+2s-\mu}} < \frac{n+2s-\mu}{2(2n-\mu)} (S^H_s)^{\frac{2n-\mu}{n+2s-\mu}}.
\end{split}
\end{equation*}
Therefore, $ c^*_{j+m}<  \frac{n+2s-\mu}{2(2n-\mu)} (S^H_s)^{\frac{2n-\mu}{n+2s-\mu}}$.\hfill{\QED}
\end{proof}

\noi \textbf{Proof of Theorem \ref{thrm2}:} Since all the conditions of Lemma \ref{br2} holds, using Lemma \ref{PSlevel} and \ref{br4}, we get the $(PS)_{c^*_{j+m}}$, for $1\leq m\leq q$. Thus, problem $(P_\la)$ has atleast $q$ distinct pairs of solution.\hfill{\QED}

\section{Regularity of weak solutions}
In this section, we prove that any weak solution of $(P_\la)$ is bounded and moreover loclly Holder continuous. First we   we prove Theorem \ref{thrm4}.
\begin{Theorem}\label{Linftes}
Let $0 \leq u \in X_0$, $n>2s$ and $\la >0$ be such that
\begin{equation*}
\begin{split}
&\int_{Q}\frac{(u(x)-u(y))(\varphi(x)-\varphi(y))}{|x-y|^{n+2s}}~\mathrm{d}x\mathrm{d}y\\
&\quad = \int_{\Om}\int_{\Om}\frac{|u(y)|^{2^*_{\mu,s}}|u(x)|^{2^*_{\mu,s}-2}u(x)\varphi(x)}{|x-y|^{\mu}}~\mathrm{d}y\mathrm{d}x+ \la \int_{\Om}u\varphi ~dx,
\end{split}
\end{equation*}
for every $\varphi \in C_c^{\infty}(\Om)$, i.e. $u$ is a nonnegative weak solution of $(P_\la)$. Then, $u \in L^{\infty}(\Om)$.
\end{Theorem}
\begin{proof}
We may assume that $u$ does not vanish identically (otherwise the proof is trivial) and let $u$ be nonnegative. Let $\delta>0$, to be chosen appropriately small whose choice will be done on \eqref{del} later in proof. Now, let $c>0$ be a constant chosen in such a way that for any $x \in \mb R^n$, $v(x) := \frac{u(x)}{c} \in X_0$ satisfies
\begin{equation}\label{reg1}
\begin{split}
&\int_{Q}\frac{(v(x)-v(y))(\varphi(x)-\varphi(y))}{|x-y|^{n+2s}}~\mathrm{d}x\mathrm{d}y\\
&\quad \leq \int_{\Om}\int_{\Om}\frac{|v(y)|^{2^*_{\mu,s}}|v(x)|^{2^*_{\mu,s}-2}v(x)\varphi(x)}{|x-y|^{\mu}}~\mathrm{d}y\mathrm{d}x+ \la \int_{\Om}v\varphi ~dx,
\end{split}
\end{equation}
for every $0 \leq \varphi \in C_c^{\infty}(\Om)$ and $|v|_{2^*_s}= \delta$. It is a simple observation that if $v \in X_0$, then $v^+ := \max\{v,0\}$ satisfies
\begin{equation}\label{reg2}
(v(x)-v(y))(v^+(x)-v^+(y)) \geq |v^+(x)-v^+(y)|^2,
\end{equation}
for any $x,y \in \mb R^n$. Let us set $C_k := 1- 2^{-k}$, $v_k := v- C_k$, $w_k := v_k^+ \in X_0$ and $U_k := |w_k|_{2^*_s}$. We get that
\[ 0 \leq |v|+ C_k \leq |v|+1 \in L^2(\Om) \subset L^{2^*_s}(\Om),\]
being $\Om$ bounded, and
\[\lim_{k \rightarrow +\infty} w_k = (v-1)^+. \]
Therefore, by the Dominated Convergence Theorem,
\begin{equation}\label{reg3}
\lim_{k \rightarrow +\infty} U_k = \left(\int_{\Om}[(v-1)^+]^{2^*_s}~\mathrm{d}x\right)^{\frac{1}{2^*_s}}.
\end{equation}
For any $k \in \mb N$, $C_{k+1}>C_k$ and so $w_{k+1} \leq w_k$ a.e. in $\mb R^n$. Also let $A_k := C_{k+1}/(C_{k+1}-C_k) = 2^{k+1}-1$, for any $k \in \mb R^n$. We claim that for any $k \in \mb N$
\begin{equation} \label{reges}
v < A_k w_k \; \text{on} \; \{w_{k+1}>0\}.
\end{equation}
To check this, let $x \in \{w_{k+1}>0\}$. Then $v(x) > C_{k+1}> C_k$, so $w_k(x)= v_k(x)=v(x)-C_k $ and
\begin{equation*}
A_k w_k(x) = v(x)+ \frac{C_k}{C_{k+1}-C_k}(v(x)- C_{k+1})> v(x).\end{equation*}
Notice also that $v_{k+1}(x)- v_{k+1}(y)= v(x)- v(y)$, for any $x,y \in \mb R^n$. Using this, \eqref{reg1}, \eqref{reges},
\eqref{reg2}, H\"{o}lder's inequality  and the fact that $w_{k+1}= v_{k+1}^+ \in X_0$, we get
\begin{align}\label{reg4}
&\int_{\mb R^{2n}}\frac{|w_{k+1}(x)- w_{k+1}(y)|^2}{|x-y|^{n+2s}}~\mathrm{d}x\mathrm{d}y = \int_{Q}\frac{|v_{k+1}^+(x)- v_{k+1}^+(y)|^2}{|x-y|^{n+2s}}~\mathrm{d}x\mathrm{d}y\notag\\
& \leq \int_{\mb R^{2n}}\frac{(v_{k+1}(x)- v_{k+1}(y))(v_{k+1}^+(x)- v_{k+1}^+(y))}{|x-y|^{n+2s}}~\mathrm{d}x\mathrm{d}y\notag\\
& = \int_{\mb R^{2n}}\frac{(v(x)- v(y))(v_{k+1}^+(x)- v_{k+1}^+(y))}{|x-y|^{n+2s}}~\mathrm{d}x\mathrm{d}y\notag\\
& \leq \int_{\Om}\int_{\Om}\frac{|v(y)|^{2^*_{\mu,s}}|v(x)|^{2^*_{\mu,s}-2}v(x)w_{k+1}(x)}{|x-y|^{\mu}}~\mathrm{d}y\mathrm{d}x+ \la \int_{\Om}v(x) w_{k+1}(x) ~dx\notag\\
& = \int_{\{w_{k+1}(x)>0\}}\int_{\Om}\frac{|v(y)|^{2^*_{\mu,s}}|v(x)|^{2^*_{\mu,s}-2}v(x)w_{k+1}(x)}{|x-y|^{\mu}}~\mathrm{d}y\mathrm{d}x+ \la \int_{\{w_{k+1}>0\}}v(x) w_{k+1}(x) ~dx\notag\\
& \leq A_k^{2^*_{\mu,s}-1} \int_{\{w_{k+1}(x)>0\}}\int_{\Om}\frac{|v(y)|^{2^*_{\mu,s}}|w_k(x)|^{2^*_{\mu,s}-1}w_{k+1}(x)}{|x-y|^{\mu}}~\mathrm{d}y\mathrm{d}x+ \la A_k \int_{\{w_{k+1}>0\}}w_k^2(x)~dx\notag\\
& \leq A_k^{2^*_{\mu,s}-1} \int_{\{w_{k+1}(x)>0\}}\int_{\Om}\frac{|v(y)|^{2^*_{\mu,s}}|w_k(x)|^{2^*_{\mu,s}}}{|x-y|^{\mu}}~\mathrm{d}y\mathrm{d}x+ \la 2^{k+1} |w_k|^2_{2^*_s}|\{w_{k+1}>0\}|^{\frac{2s}{n}}~dx.
\end{align}
Let us consider the first integral of R.H.S. of above inequality separately and we get that
\begin{equation}\label{reg5}
\begin{split}
&\int_{\{w_{k+1}(x)>0\}}\int_{\Om}\frac{|v(y)|^{2^*_{\mu,s}}|w_k(x)|^{2^*_{\mu,s}}}{|x-y|^{\mu}}~\mathrm{d}y\mathrm{d}x\\
& \leq \left(\int_{\{w_{k+1}(x)>0\}}\int_{\{v(y)\leq C_{k+1}\}} +
\int_{\{w_{k+1}(x)>0\}}\int_{\{v(y)>C_{k+1}\}}\right)\frac{|v(y)|^{2^*_{\mu,s}}|w_k(x)|^{2^*_{\mu,s}}}{|x-y|^{\mu}}~\mathrm{d}y\mathrm{d}x\\
& = I_1+ I_2, \; \text{(say)}.
\end{split}
\end{equation}
Now using \eqref{reges} and Hardy- Littlewood- Sobolev inequality, we have
\begin{equation}\label{reg6}
\begin{split}
I_1=& \int_{\{w_{k+1}(x)>0\}}\int_{\{w_{k+1}(y)>0\}}\frac{|v(y)|^{2^*_{\mu,s}}|w_k(x)|^{2^*_{\mu,s}}}{|x-y|^{\mu}}~\mathrm{d}y\mathrm{d}x\\
&\leq A_k^{2^*_{\mu,s}} \frac{|w_k(y)|^{2^*_{\mu,s}}|w_k(x)|^{2^*_{\mu,s}}}{|x-y|^{\mu}}~\mathrm{d}y\mathrm{d}x \leq A_k^{2^*_{\mu,s}} C(n,\mu) |w_k|_{2^*_s}^{22^*_{\mu,s}}.
 \end{split}
\end{equation}
Next, again using \eqref{reges} and H\"{o}lder's inequality we have
\begin{equation}\label{reg7}
\begin{split}
I_2&=  \int_{\{w_{k+1}(x)>0\}}\int_{\{w_{k+1}(y)>0\}}\frac{|v(y)|^{2^*_{\mu,s}}|w_k(x)|^{2^*_{\mu,s}}}{|x-y|^{\mu}}~\mathrm{d}y\mathrm{d}x\\
& \leq C_{k+1}^{2^*_{\mu,s}} \int_{\{w_{k+1}(x)>0\}} |w_k(x)|^{2^*_{\mu,s}} \int_{\Om}\frac{\mathrm{d}y}{|x-y|^{\mu}}~\mathrm{d}x\\
& \leq M  C_{k+1}^{2^*_{\mu,s}} \int_{\{w_{k+1}(x)>0\}} |w_k(x)|^{2^*_{\mu,s}}\mathrm{d}x\\
& \leq M  C_{k+1}^{2^*_{\mu,s}} |\{w_{k+1}>0\}|^{\frac{\mu}{2n}} |w_k|_{2^*_s}^{2^*_{\mu,s}}.
\end{split}
\end{equation}
Using \eqref{reg5}, \eqref{reg6}, \eqref{reg7} and Sobolev inequality in \eqref{reg4}, we get
\begin{equation}\label{reg8}
\begin{split}
S_s |w_{k+1}|_{2^*_s}^2 &\leq \int_{\mb R^{2n}}\frac{|w_{k+1}(x)- w_{k+1}(y)|^2}{|x-y|^{n+2s}}~\mathrm{d}x\mathrm{d}y \\
& \leq  A_k^{2^*_{\mu,s}-1} \left( A_k^{2^*_{\mu,s}} C(n,\mu) |w_k|_{2^*_s}^{22^*_{\mu,s}}+
 M  C_{k+1}^{2^*_{\mu,s}} |\{w_{k+1}>0\}|^{\frac{\mu}{2n}} |w_k|_{2^*_s}^{2^*_{\mu,s}}\right.\\
 & \quad \left.+ 2^{k+1} |w_k|^2_{2^*_s}|\{w_{k+1}>0\}|^{\frac{2s}{n}}~dx\right).
\end{split}
\end{equation}
Now we claim that
\begin{equation}\label{reg9}
\{w_{k+1}>0\} \subset \{w_k > 2^{-(k+1)}\}.
\end{equation}
To establish this, we observe that if $x \in \{w_{k+1}>0\}$ then
\[v(x)-C_{k+1}>0.\]
Accordingly, $v_k(x) = v(x)- C_{k}> C_{k+1}- C_k = 2^{-(k+1)}$, so that, $$w_k(x)= v_k(x)> 2^{-(k+1)}.$$ Thus, \eqref{reg9} gives
 \begin{equation}\label{reg10}
 U_k^{2^*_s} = |w_k|_{2^*_s} \geq \int_{\{w_k > 2^{-(k+1)}\}}w_k^{2^*_s}~\mathrm{d}x
 \geq 2^{-(k+1)}|\{w_{k+1}>0\}|.
 \end{equation}
 As a consequence of \eqref{reg10}, from \eqref{reg8} we get
\begin{equation}\label{reg11}
\begin{split}
S_s |w_{k+1}|_{2^*_s}^2 &\leq  A_k^{2^*_{\mu,s}-1}  \left( A_k^{2^*_{\mu,s}} C(n,\mu) |w_k|_{2^*_s}^{22^*_{\mu,s}}+
 M  C_{k+1}^{2^*_{\mu,s}} 2^{\frac{\mu(k+1)}{2n}} |w_k|_{2^*_s}^{2^*_{s}}\right.\\
 & \quad \left.+ 2^{k+1} |w_k|^{2^*_s}_{2^*_s}2^{\frac{2s(k+1)}{n}}~dx\right)\\
 & \leq  2^{(2^*_{\mu,s}-1)(k+1)}  \left( 2^{2^*_{\mu,s}(k+1)} C(n,\mu) |w_k|_{2^*_s}^{22^*_{\mu,s}}+
 M 2^{\frac{\mu(k+1)}{2n}} |w_k|_{2^*_s}^{2^*_{s}}\right.\\
 & \quad \left.+ 2^{k+1} |w_k|^{2^*_s}_{2^*_s}2^{\frac{2s(k+1)}{n}}~dx\right)\\
 & \leq 2^{(2^*_{\mu,s}-1)(k+1)}\max\{2^{2^*_{\mu,s}(k+1)} C(n,\mu), M 2^{\frac{\mu(k+1)}{2n}}+ 2^{(k+1)(1+\frac{2s}{n})}\}\times\\
 & \quad \left( |w_k|_{2^*_s}^{22^*_{\mu,s}}+ |w_k|^{2^*_s}_{2^*_s} \right).
 \end{split}
\end{equation}
Therefore using definition of $U_k$ in \eqref{reg11}, we get
\begin{equation}\label{reg12}
U_{k+1} \leq  D^{(k+1)} \left(  U_k^{2^*_{\mu,s}}+ U_k^{\frac{2^*_s}{2}}\right),
\end{equation}
where, $D = \left(1+ (2^{(2^*_{\mu,s}-1)}\max\{2^{2^*_{\mu,s}} C(n,\mu), M 2^{\frac{\mu}{2n}}+ 2^{(1+\frac{2s}{n})}\})^{1/2}\right)>1$ and $2^*_{\mu,s}>{2^*_s}/{2} >1 $.\\
Now we are ready to perform our choice of $\delta$: namely we assume that $\delta >0$ is so small that
\begin{equation}\label{del}
\delta^{\frac{2^*_s}{2}-1} < \frac{1}{(2^{2^*_{\mu,s}}D)^{\frac{1}{({2^*_s}/{2})-1}}}.
\end{equation}
We also fix $\eta \in \left( \delta^{\frac{2^*_s}{2}-1}, \frac{1}{(2^{2^*_{\mu,s}}D)^{\frac{1}{({2^*_s}/{2})-1}}}  \right)$. Since $D>1$ and $2^*_s/2 >1$, we get $\eta \in (0,1)$. Moreover,
\begin{equation}\label{reg13}
\delta^{\frac{2^*_s}{2}-1} \leq \eta \; \; \text{and}\; \; 2^{2^*_{\mu,s}}D \eta^{\frac{2^*_s}{2}-1} \leq 1.
\end{equation}
We claim that
\begin{equation}\label{reg14}
U_k \leq 2\delta \eta^{k+1}.
\end{equation}
The proof is by induction. First of all
\[ U_0 = |v^+|_{2^*_s} \leq |v|_{2^*_s} = \delta \leq 2\delta \]
which is \eqref{reg14} when $k=0$. Let us now suppose that \eqref{reg14}  holds true for $k$ and let us prove it for $k+1$. Using \eqref{reg12} and \eqref{reg13}, we get
\begin{equation*}
\begin{split}
U_{k+1} &\leq  D^{k+1}( U_k^{2^*_{\mu,s}}+ U_k^{\frac{2^*_s}{2}}) \leq 2^{2^*_{\mu,s}+1}D^{k+1} (\delta \eta^{k+1})^{\frac{2^*_s}{2}}\\
& \leq 2\delta (2^{2^*_{\mu,s}}D\eta^{\frac{2^*_s}{2}-1})^{k+1} \delta^{\frac{2^*_s}{2}-1}\eta^k  \leq 2\delta \eta^{k+1}.
\end{split}
\end{equation*}
This proves our claim \eqref{reg14}. Then using $\eta \in (0,1)$ and \eqref{reg14}, we conclude that
\[\lim_{k \rightarrow +\infty} U_k=0.\]
Hence, by \eqref{reg3}, $(v-1)^+ =0$ a.e. in $\Om$, that is $v \leq 1$ a.e. in $\Om$. Therefore, $u \leq c$ a.e. in $\Om$ which implies $|u|_{\infty} \leq c$. This completes the proof.\QED
\end{proof}

\begin{Theorem}
Let $u$ be a positive solution of $(P_{\la})$. Then there exist $\alpha \in (0,s]$ such that $u \in C_{loc}^{\alpha}(\Om)$..
\end{Theorem}
\begin{proof}
Let $\Om^{\prime} \in \Om$. Then using above regularity result, for any $\psi \in C_{c}^{\infty}(\Om)$ we obtain
\begin{equation*}
\int_{\Om^\prime}\int_{\Om^\prime}\frac{|v(y)|^{2^*_{\mu,s}}|v(x)|^{2^*_{\mu,s}-2}v(x)\psi(x)}{|x-y|^{\mu}}~\mathrm{d}y\mathrm{d}x+ \la \int_{\Om^\prime}v\psi ~dx\leq C \int_{\Om^{\prime}}\psi dx
\end{equation*}
for some constant $C>0$, since $u \in L^{\infty}(\Om)$. Thus we have $|(-\De_{p})^su|\leq C$ weakly on $\Om^{\prime}$. So, using theorem 4.4 of \cite{Asm} and applying a covering argument on inequality in corollary 5.5 of \cite{Asm}, we can prove that there exist $\alpha \in (0,s] $ such that $u \in C^{\alpha}(\Om^{\prime})$, for all $\Om^{\prime} \Subset \Om$. Therefore, $u \in C_{loc}^{\alpha}(\Om)$. \QED
\end{proof}

\section{Nonexistence result}
In this section, we prove a non-existence result for $\la \le 0$ when $\Omega$ is a star shaped domain. At first, we prove the Pohozaev type identity:

\begin{Proposition}\label{Poho}
If $n>2s$, $\la <0$, $\Om$ be bounded, $C^{1,1}$ domain and $u \in L^{\infty}(\Om)$ solves $(P_\la)$, then
\begin{equation*}
\begin{split}
&\frac{2s-n}{2}\int_{\Om}u(-\De)^su~\mathrm{d}x - \frac{\Gamma(1+s)^2}{2}\int_{\partial \Om}\left(\frac{u}{\delta^s}\right)^2(x.\nu)\mathrm{d}\sigma\\
 &=\frac{2n-\mu}{22^*_{\mu,s}}\int_{\Om}\int_{\Om}\frac{|u(x)|^{2^*_{\mu,s}}|u(y)|^{2^*_{\mu,s}}}{|x-y|^{\mu}}~\mathrm{d}x\mathrm{d}y+ \frac{\la}{n} \int_{\Om}|u|^2\mathrm{d}x,
 \end{split}
 \end{equation*}
where $\delta(x)=  \text{dist }(x,\partial \Om)$ and $\nu$ denotes the unit outward normal to $\partial \Om$ at $x$ and $\Gamma$ is the Gamma function.
\end{Proposition}
\begin{proof}
Since $u$ solves $(P_\la)$, u satisfies the problem
\[(-\De)^su = \left( \int_{\Om}\frac{|u(y)|^{2^*_{\mu,s}}}{|x-y|^{\mu}}\mathrm{d}y \right)|u|^{2^*_{\mu,s}-2}u +\la u \; \text{ in } \Om.\]
Multiplying both sides of the above equation by $(x. \nabla u)$ and integrating, we get
\begin{equation}\label{poho1}
\int_{\Om}(x. \nabla u)(-\De)^su~ \mathrm{d}x= \int_{\Om}(x.\nabla u)\left(\int_{\Om}\frac{|u(y)|^{2^*_{\mu,s}}}{|x-y|^{\mu}}\mathrm{d}y \right)|u|^{2^*_{\mu,s}-1}\mathrm{d}x+ \frac{\la n}{2}\int_{\Om}|u|^2\mathrm{d}x.
\end{equation}
Using Theorem $1.4$ and $1.6$ of \cite{fracpoho} (since $u \in L^{\infty}(\Om)$, $f(u) \in C^{0,1}(\bar \Om)$, where $f(u)= \left( \int_{\Om}\frac{|u|^{2^*_{\mu,s}}}{|x-y|^{\mu}}\mathrm{d}y \right)|u|^{2^*_{\mu,s}-2}u +\la u$), we get
\begin{equation*}
\int_{\Om}(x. \nabla u)(-\De)^su~ \mathrm{d}x= \frac{2s-n}{2}\int_{\Om}u(-\De)^su~\mathrm{d}x - \frac{\Gamma(1+s)^2}{2}\int_{\partial \Om}\left(\frac{u}{\delta^s}\right)^2(x.\nu)\mathrm{d}\sigma.
\end{equation*}
Now, consider the term
\begin{equation*}
\begin{split}
 &\int_{\Om}(x.\nabla u)\left(\int_{\Om}\frac{|u(y)|^{2^*_{\mu,s}}}{|x-y|^{\mu}}\mathrm{d}y \right)|u|^{2^*_{\mu,s}-1}\mathrm{d}x\\
 & = -\int_{\Om}u(x) \left(n \left(\int_{\Om}\frac{|u(y)|^{2^*_{\mu,s}}}{|x-y|^{\mu}}\mathrm{d}y \right)|u|^{2^*_{\mu,s}-1} + (2^*_{\mu,s}-1)|u(x)|^{2^*_{\mu,s}-2}x.\nabla u(x)\int_{\Om}\frac{|u(y)|^{2^*_{\mu,s}}}{|x-y|^{\mu}}\mathrm{d}y\right.\\
 & \quad \quad +\left. |u(x)|^{2^*_{\mu,s}-1}\int_{\Om}(-\mu)x.(x-y)\frac{|u(y)|^{2^*_{\mu,s}}}{|x-y|^{\mu+2}}\mathrm{d}y \right) \mathrm{d}x\\
 & = - n \int_{\Om}\int_{\Om}\frac{|u(x)|^{2^*_{\mu,s}}|u(y)|^{2^*_{\mu,s}}}{|x-y|^{\mu}}~\mathrm{d}x\mathrm{d}y - (2^*_{\mu,s}-1)\int_{\Om}x.\nabla u(x)\int_{\Om}\frac{|u(x)|^{2^*_{\mu,s}}}{|x-y|^{\mu}}\mathrm{d}y|u(x)|^{2^*_{\mu,s}-1}\mathrm{d}x\\
 & \quad \quad+ \mu \int_{\Om}\int_{\Om}x.(x-y) \frac{|u(y)|^{2^*_{\mu,s}}}{|x-y|^{\mu+2}}|u(x)|^{2^*_{\mu,s}}\mathrm{d}y\mathrm{d}x.
 \end{split}
\end{equation*}
This gives
\begin{equation*}
\begin{split}
& 2^*_{\mu,s}\int_{\Om}x.\nabla u(x)\int_{\Om}\frac{|u(y)|^{2^*_{\mu,s}}}{|x-y|^{\mu}}\mathrm{d}y|u(x)|^{2^*_{\mu,s}-1}\mathrm{d}x\\
& =\quad  - n \int_{\Om}\int_{\Om}\frac{|u(x)|^{2^*_{\mu,s}}|u(y)|^{2^*_{\mu,s}}}{|x-y|^{\mu}}~\mathrm{d}x\mathrm{d}y + \mu \int_{\Om}\int_{\Om}x.(x-y) \frac{|u(y)|^{2^*_{\mu,s}}}{|x-y|^{\mu+2}}|u(x)|^{2^*_{\mu,s}}\mathrm{d}y\mathrm{d}x,
\end{split}
\end{equation*}
and similarly,
\begin{equation*}
\begin{split}
& 2^*_{\mu,s}\int_{\Om}y.\nabla u(x)\int_{\Om}\frac{|u(x)|^{2^*_{\mu,s}}}{|x-y|^{\mu}}\mathrm{d}x|u(y)|^{2^*_{\mu,s}-1}\mathrm{d}y\\
& =  - n \int_{\Om}\int_{\Om}\frac{|u(y)|^{2^*_{\mu,s}}|u(x)|^{2^*_{\mu,s}}}{|x-y|^{\mu}}~\mathrm{d}y\mathrm{d}x + \mu \int_{\Om}\int_{\Om}y.(y-x) \frac{|u(x)|^{2^*_{\mu,s}}}{|x-y|^{\mu+2}}|u(y)|^{2^*_{\mu,s}}\mathrm{d}x\mathrm{d}y.
\end{split}
\end{equation*}
Thus, we have
\begin{equation*}
\int_{\Om}(x.\nabla u(x))\left(\int_{\Om}\frac{|u(y)|^{2^*_{\mu,s}}}{|x-y|^{\mu}}\mathrm{d}y\right)|u(x)|^{2^*_{\mu,s}-1}\mathrm{d}x=
\frac{\mu -2n}{2 2^*_{\mu,s}}\int_{\Om}\frac{|u(x)|^{2^*_{\mu,s}}|u(x)|^{2^*_{\mu,s}}}{|x-y|^{\mu}}\mathrm{d}x\mathrm{d}y.
\end{equation*}
Since \[\int_{\Om}(x.\nabla u)u \mathrm{d}x= -\frac{n}{2}\int_{\Om}u^2 \mathrm{d}x,\] using \eqref{poho1}, the result follows. \hfill{\QED}
\end{proof}

\noi \textbf{Proof of Theorem \ref{thrm3}:} Let $u \geq 0 $ be a nontrivial solution of $(P_{\la})$, then by Theorem \ref{Linftes}, $u \in L^{\infty}(\Om)$. Therefore, we have
\[\|u\|^2 = \int_{\Om}\int_{\Om}\frac{|u(x)|^{2^*_{\mu,s}}|u(y)|^{2^*_{\mu,s}}}{|x-y|^{\mu}}~\mathrm{d}x\mathrm{d}y+ \la\int_{\Om}u^2\mathrm{d}x. \]
Using Proposition \ref{Poho}, we get
\[  \frac{\Gamma(1+s)^2}{2}\int_{\partial \Om}\left(\frac{u}{\delta^s}\right)^2(x.\nu)\mathrm{d}\sigma= \la\int_{\Om}u^2\mathrm{d}x.\]
But, since $\Om$ is star shaped with respect to origin in $\mb R^n$, so $x. \nu >0$ . From above equation and $\la <0$, we have $u \equiv 0$, which is a contradiction. This completes the proof. \hfill{\QED}

 \linespread{0.5}

\end{document}